\documentclass[a4paper]{amsart}
\usepackage{a4}
\usepackage{amsmath}
\usepackage{amssymb}
\usepackage[applemac]{inputenc}
\usepackage{graphicx}
\usepackage{color}
\newcommand{\ds}{\displaystyle}
\newcommand{\supp}{\mathrm{supp}\;}

\newtheorem{remark}{\textbf{Remark}}[section]

\newtheorem{lemma}{\textbf{Lemma}}[section]
\newtheorem{theorem}{\textbf{Theorem}}[section]
\newtheorem{corollary}{\textbf{Corollary}}[section]
\newtheorem{proposition}{\textbf{Proposition}}[section]

\newtheorem{definition}{\textbf{Definition}}[section]

\numberwithin{equation}{section}

\title[On the asymptotic nature of first order \mbox{MFGs}]{On the asymptotic nature of first order mean field games} 

\author{Markus Fischer \and Francisco J. Silva}

\thanks{Dipartimento di Matematica ``Tullio Levi-Civita'', Universit{\`a} degli Studi di Padova, via Trieste, 63, 35121 Padova, Italia (fischer@math.unipd.it)}

\thanks{Institut de recherche XLIM-DMI, UMR-CNRS 7252, Facult\'e des Sciences et Techniques, 
Universit\'e de Limoges, 87060 Limoges, France (francisco.silva@unilim.fr)}

\def\dd{{\rm d}}

\def\weight(#1,#2){c_{#1,#2}}

 \if {
  
 } \fi

\if {

}{{\caption}}
} \fi

\def\A{\mathcal{A}}
\def\B{\mathcal{B}}
\def\C{\mathcal{C}}

\def\F{\mathcal{F}}

\def\L{\mathcal{L}}

\def\P{\mathcal{P}}

\def\SS{\mathcal{S}}

\def\eps{\varepsilon}

\def\supp{\mathop{\rm supp}}

\def\1B{{\bf  1}}

\newcommand{\NN}{\mathbb{N}}

\newcommand{\OO}{\mathcal{O}}
\newcommand{\RR}{\mathbb{R}}

\def\II{\mathbb{I}}

\newcommand\be{\begin{equation}}
\newcommand\ee{\end{equation}}
\newcommand\ba{\begin{array}}
\newcommand\ea{\end{array}}
\newcommand{\bean}{\begin{eqnarray*}}
\newcommand{\eean}{\end{eqnarray*}}

\def\ds{\displaystyle}

\begin{document}

\begin{abstract}
For a class of finite horizon first order mean field games and associated $N$-player games, we give a simple proof of convergence of symmetric $N$-player Nash equilibria in distributed open-loop strategies to solutions of the mean field game in Lagrangian form. Lagrangian solutions are then connected with those determined by the usual mean field game system of two coupled first order \mbox{PDEs}, and convergence of Nash equilibria in distributed Markov strategies is established.
\end{abstract}

\maketitle

{\small
\noindent {\bf AMS-Subject Classification:}  49N70, 60B10, 91A06, 91A13. \\[0.5ex]

}

\section{Introduction} \label{Sect_Intro}

The purpose of this article is to illustrate a simple way of establishing convergence of open-loop Nash equilibria in the case of first-order non-stationary Mean Field Games (MFGs). Introduced by J.-M.~Lasry and P.-L.~Lions and, independently, by M.~Huang, R.P.~Malham{\'e} and P.E.~Caines about fifteen years ago (cf.\ \cite{LasryLions07, huangetal06}), mean field games are limit models for non-cooperative symmetric $N$-player  differential  games as the number of players $N$ tends to infinity; see, for instance, the lecture notes \cite{Cardaliaguet10} and the recent two-volume work \cite{carmonadelarue}. The notion of solution usually adopted for the prelimit models is that of a Nash equilibrium. A standard way of making the connection with the limit model rigorous is to show that a solution of the mean field game yields approximate Nash equilibria for the $N$-player games, with approximation error vanishing as $N\to \infty$. In the opposite direction, one aims to prove that a sequence of $N$-player Nash equilibria converges, as $N$ tends to infinity, to the mean field game limit. 

When Nash equilibria are considered in stochastic open-loop strategies, then  their  convergence is well understood and can be established under mild conditions; see \cite{fischer17} and \cite{lacker16}, both for finite horizon games with general, possibly degenerate, Brownian dynamics. The convergence  analysis is much harder when Nash equilibria are defined over Markov feedback strategies with full state information.

A first result in this setting was given by Gomes, Mohr, and Souza \cite{gomes} for continuous time games with finite state space. There, convergence of Markovian Nash equilibria is proved, but only if the time horizon is small enough. A breakthrough was achieved by Cardaliaguet, Delarue, Lasry, and Lions in \cite{cardaliaguetetal15}. In the setting of games with non-degenerate Brownian dynamics, possibly including common noise, convergence to the mean field game limit is established there for arbitrary time horizon provided the so-called master equation associated with the mean field game possesses a unique sufficiently regular solution. In this case, the convergence analysis can be refined, yielding not only convergence of minimal costs and propagation of chaos for the Nash equilibrium state trajectories, but also fluctuation and large deviations results for the associated empirical measures; see Cecchin and Pelino \cite{cp} and, independently, Bayraktar and Cohen \cite{bayraktarcohen} for finite state games and the papers by Delarue, Lacker, and Ramanan \cite{delarueetala, delarueetalb} for Brownian dynamics without or with common noise.

Well-posedness of the master equation implies uniqueness of solutions for the mean field game. But also the situation where the mean field game possesses multiple solutions, while the $N$-player Nash equilibria in full Markov feedback strategies are still uniquely determined, occurs. In this case, the convergence problem is in part open. The most general result appears to be the recent preprint \cite{lacker} by Lacker. There, for a class of games with non-degenerate Brownian dynamics, it is shown that all limit points of the $N$-player Nash equilibria are concentrated on weak solutions of the mean field game; these are more general than randomizations of ordinary (``strong'') solutions of the mean field game. In the three recent works \cite{nutzetalii}, \cite{delaruetchuendom} and \cite{cdfp}, the authors present case studies, giving criteria for characterizing those mean field game solutions that can be attained as limits of feedback Nash equilibria with full state information. 

Here, we consider a much simpler situation: The underlying dynamics are deterministic with direct control of players' states; randomness enters only through the players' initial positions, which are assumed to be independently and identically distributed. Thanks to the deterministic dynamics without explicit interaction, players can directly control their entire state trajectories. Thus, the set of strategies (or actions) of each player consists of trajectory-valued functions depending on time and player's own position.   We call these actions {\it distributed open-loop strategies}.     By considering randomizations of these strategies, obtaining what we call {\it randomized  distributed open-loop strategies},  we obtain existence of symmetric Nash equilibria in this new class of actions (through a standard fixed point argument), as well as precompactness of sequences of such equilibria. Convergence to the mean field game equilibrium in Lagrangian form (see e.g.\ \cite{CH17,MR3644590,cardaliaguet_meszaros_santambrogio_2018,cannarsa_capuani_2018})) along weakly converging subsequences of Nash equilibria is then established in Theorem  \ref{main_result} by showing that the variance of the associated empirical measures (evaluated by integrating over test functions from a countable family) tends to zero as $N\to \infty$. Thus, to obtain convergence, we essentially prove a law of large numbers for a triangular array $(Y^{N}_{i})_{i\in \{1,\ldots,N\}, N\in \mathbb{N}}$ where $Y^{N}_{1},\ldots,Y^{N}_{N}$ are independent and identically distributed with common marginal law that however depends on $N$.  Let us point out that after obtaining a suitable compactness property for the set of admissible strategies, the convergence result in Theorem  \ref{main_result} can actually be deduced from the more general results in 
\cite{fischer17} or \cite{lacker16}. Thus, the main purpose of this result is therefore to give a simple proof of convergence, which takes advantage of the structure of the dynamics and the deterministic nature of the underlying differential games.  We believe that the techniques considered here could be useful in order  to justify the asymptotic nature of more sophisticated {\it deterministic} MFGs (see e.g. \cite{cannarsa_capuani_2018,Cannarsa_Capuani_Cardaliaguet_18} dealing with state constrained problems).

In the second part of this article, we consider additional second order assumptions on the data and we assume that the initial distribution is absolutely continuous with respect to the Lebesgue measure. In this framework, and following  \cite{CH17}, we link the notion of Lagrangian MFG equilibrium to the well known \mbox{PDE} characterization of mean field games in terms of two coupled first order partial differential equations, namely a backward first order Hamilton-Jacobi-Bellman equation and a  forward continuity equation; see equation $(MFG)$ in Subsection \ref{equivalence_both_mfg_formulations} below. Under the stronger assumptions mentioned above, for each $N \in \NN$, any symmetric randomized distributed open-loop Nash equilibrium of the $N$-player game can be identified with a symmetric distributed open-loop Nash equilibrium (non-randomized). Moreover, the strategies associated to such equilibria can be described by controls which are feedback with respect to the individual  states.  We call these actions {\it distributed Markov strategies}. The latter are determined by the solutions of   
%
 a coupled system of first order PDEs analogous to the mean field game \mbox{PDE} system; see Eq.~$(MFG_{N})$ in Subsection~\ref{PDE_N_players}. This is in contrast with $N$-player Nash equilibria defined over Markov feedback strategies with full state information, which are determined through a system of $N$ coupled \mbox{PDEs}. Building on the equivalence of characterizations and the convergence result in Theorem  \ref{main_result}, we also establish in Theorem \ref{convergence_result_regular_case}  the  convergence of solutions of  $(MFG_{N})$  to solutions of $(MFG)$, as well as the convergence of the Nash equilibria in distributed Markov  strategies.

The rest of this paper is organized as follows. In Section~\ref{n_player_game}, we introduce the $N$-player games together with some notation and our standing assumptions. Existence of Nash equilibria in randomized distributed open-loop strategies is verified in Proposition~\ref{nash_equilibria_finite_game}. In Section~\ref{sect_mfg}, the associated mean field game is introduced in Lagrangian form; see Definition~\ref{MFG_equilibrium} and Proposition~\ref{equivalent_reformulation_mfg_eq}. We show convergence of symmetric $N$-player randomized distributed open-loop Nash equilibria to the mean field game limit in Theorem~\ref{main_result}. Section~\ref{sect_pde_approach} links, under additional assumptions, the mean field game in Lagrangian form with the mean field game \mbox{PDE} system. Similarly, $N$-player Nash equilibria in   distributed Markov strategies are described in terms of a \mbox{PDE} system analogous to the limit system, but dependent on the number of players $N$. The corresponding convergence results are given in Theorem~\ref{convergence_result_regular_case} and Corollary~\ref{Convergence_from_the_system}, respectively.

%

\section{The $N$-player game}\label{n_player_game}

Before introducing the $N$-player deterministic differential game that we are interested in, let us first fix some  notations. Given a Polish space $(X,d_X)$, we denote by $\P(X)$ the set of probability measures over $X$. If $(Y,d_Y)$ is another Polish space and $\varphi: X \to Y$ a Borel measurable function, then the {\it push-forward} $\varphi \sharp m \in \P(Y)$ of a measure $m \in \P(X)$ by $\varphi$ is defined by 
\be\label{definition_push_forward}
\varphi \sharp m (A):= m(\varphi^{-1}(A)) \hspace{0.5cm} \forall \; A \in \B(Y), 
\ee
where  $ \B(Y)$ denotes the $\sigma$-algebra of Borel sets in $Y$. Given $p \in [1,+\infty)$ we denote by $\P_{p}(X)$ the set of probability measures $\mu$ over $X$ such that $\int_{X} d_X(x,x_0)^p \dd \mu(x) < +\infty$ for some $x_0 \in X$. The set $\P_{p}(X)$ is endowed with the Monge-Kantorovich distance $d_p : \P_{p}(X) \times \P_{p}(X) \to [0,\infty)$ defined by
\[
	d_p(\mu_1, \mu_2):= \inf \left\{\left(\int_{X \times X} d_{X}^p(x,y) \dd \gamma(x,y)\right)^{\frac{1}{p}} \; | \; \gamma \in \P(X \times X), \; \pi_{1}\sharp \gamma= \mu_1, \; \; \pi_{2}\sharp \gamma= \mu_2  \right\},
\]
where $\pi_{i}:X \times X \to X$ ($i=1$, $2$) is the projection on the $i$-th coordinate, that is, $\pi_{i}(x_1,x_2):=x_{i}$. It is well known (see e.g. \cite{Villani03}) that in the particular case $p=1$ we have  
\be\label{monge_rubinsteint_form_d_1} d_1(\mu_1, \mu_2)= \sup \left\{ \int_{X}f(x) \dd \left(\mu_1-\mu_2\right)(x) \; | \; f \in \mbox{Lip}_1(X) \right\}, \ee
where $\mbox{Lip}_1(X)$ denotes the set of Lipschitz functions on $X$  with Lipschitz constant equal to one. 

We will fix as data some functions $\ell: \RR^d \times \RR^d \times \P_1(\RR^d)\to \RR$, $\Phi: \RR^d \times \P_1(\RR^d) \to \RR$ and a probability measure $m_0 \in \P(\RR^d)$. The functions $\ell$, $\Phi$ will represent running and terminal costs, respectively, while $m_{0}$ will be an initial state distribution. We will assume that:

\medskip\noindent
{\bf(A1)} {\rm(i)} The functions $\ell$ and $\Phi$ are continuous. Moreover, the following properties hold true: 
\begin{itemize}
\item[{\rm(i.1)}] For every $(x, \mu) \in \RR^{d} \times \P_1(\RR^d)$, the function $\ell( \cdot,x, \mu)$ is convex.  
\item[{\rm(i.2)}]  There exist $q>1$, $\underline{c}_\ell>0$, $\overline{c}_\ell>0$ and $C_\ell>0$ such that 
\be\label{bounds_for_ell}
\underline{c}_\ell |\alpha|^{q}-C_\ell \leq \ell( \alpha,x, \mu) \leq \overline{c}_\ell |\alpha|^{q}+C_\ell \hspace{0.5cm} \forall \;  \alpha \in \RR^d, \; x \in \RR^{d}, \; \; \mu \in \P_1(\RR^d),
\ee
where we denote by $|\cdot|$ the Eulidean norm in $\RR^d$.
\item[{\rm(i.3)}] The function $\Phi$ is bounded from below. 
\end{itemize}
{\rm(ii)}  The support of $m_0$,  denoted by $\mbox{supp}(m_0)$, is a compact subset of $\RR^d$. \smallskip\\


Choose $T > 0$, the finite time horizon, and set $\Gamma:= C([0,T];\RR^d)$, the space of continuous $\RR^d$-valued trajectories on $[0,T]$. The space $\Gamma$ is naturally endowed with the topology of uniform convergence. Let $W^{1,q}([0,T]; \RR^d)$ denote the Sobolev space of $\RR^d$-valued absolutely continuous functions that possess first order weak sense partial derivatives in $L^{q}((0,T); \RR^d)$.

In order to introduce the game that we will consider, assume first that there are $N$ individuals ($N\geq 2$), which, from now on, will be called {\it players}, positioned at $x_1, \hdots, x_{N} \in \RR^d$  at time $t=0$. For $i \in \{1, \hdots, N\}$, player $i$ chooses a trajectory $\gamma_i \in \A(x_i)$, where
\[
	\mathcal{A}(x):= \left\{ \gamma \in  \Gamma \; \big| \;  \gamma \in W^{1,q}([0,T]; \RR^d)   \; \; \mbox{and } \;  \gamma(0)=x\right\},\quad x\in \RR^d.
\]
Given these initial positions, the cost  $j_{i}^{N}: \prod_{j=1}^{N} \A(x_j) \to \RR$ for player $i$  is defined by 
\be\label{costs_fixed_initial_conditions}j_{i}^{N}(\gamma_1, \hdots,\gamma_N):= \int_{0}^{T}  \ell\left( \dot{\gamma}_i(t),\gamma_i(t),(\gamma_{j}(t))_{j\neq i}\right)  \dd t +  \Phi\left( \gamma_{i}(T), (\gamma_{j}(T))_{j\neq i} \right),
\ee
where, for notational convenience, we have denoted 
\[
	\ell \left( \dot{\gamma}_i(t), \gamma_i(t), (\gamma_{j}(t))_{j\neq i} \right):= \ell \left( \dot{\gamma}_i(t), \gamma_i(t), \frac{1}{N-1}\sum_{_{j\neq i}} \delta_{\gamma_{j}(t)}\right), \hspace{0.2cm} \mbox{with the same convention for $\Phi$.}
\]
Assumption {\bf(A1)} implies that $j_{i}^{N}$ is well-defined.  Note that defining 
\be\label{definition_small_j_N}j^{N}(\gamma_1; (\gamma_{j})_{j=2}^{N}):=\int_{0}^{T} \ell\left(\dot{\gamma}_1(t), \gamma_1(t), (\gamma_{j}(t))_{j=2}^{N} \right)  \dd t +  \Phi\left( \gamma_{1}(T), (\gamma_{j}(T))_{j=2}^{N}\right),  
\ee
we have that $j_{i}^{N}(\gamma_1, \hdots,\gamma_N)=j^{N}(\gamma_i; (\gamma_{j})_{j\neq i})$. 

In the game that we will consider, the initial position of each player is random, independent of the initial positions of the other players and with the same law $m_0$. 
In this new context, we  define the set of {\it distributed open-loop strategies}  of the players  as 
\[
	\mathcal{A}:= \left\{ \gamma: \RR^d \to   \Gamma  \; \big| \;  \gamma \text{ is Borel measurable and } \gamma(x) \in \A(x) \text{ for all } x\in \mbox{supp}(m_0)\right\}.
\]
For notational simplicity, for $\gamma \in \A$, we will write $\gamma^x:=\gamma(x)$. Given a profile of actions $(\gamma_1, \hdots, \gamma_N)\in \A^N$, it is  natural to define the cost that it induces on player $i$ as the mean with respect to the initial conditions  of costs having the form \eqref{costs_fixed_initial_conditions}. Namely, the \emph{cost function} $J_{i}^N: \A^N \to \RR$ for player $i$ is defined as
\[
	J_i^{N}(\gamma_1, \hdots, \gamma_N):= \int_{(\RR^d)^{N}} j_i^N(\gamma^{x_1}_{1}, \hdots, \gamma^{x_N}_{N})  \otimes_{j=1}^{N}  \dd m_{0} (x_{j}).
\]
Recalling \eqref{definition_small_j_N} and defining  
\be\label{definition_J_N}
J^{N}(\gamma_{1}; (\gamma_{j})_{j=2}^{N} ):= \int_{(\RR^{d})^{N}}j^{N}(\gamma_1^{x_1}; (\gamma_{j}^{x_j})_{j=2}^{N})\otimes_{j=1}^{N}  \dd m_{0} (x_{j}),
\ee
we have $J_i^{N}(\gamma_1, \hdots, \gamma_N)=J^{N}(\gamma_{i}; (\gamma_{j})_{j\neq i} )$ for all $i=1, \hdots, N$. In particular, the costs are {\it symmetric}, i.e.\ for every permutation $\ds \sigma: \{1,\hdots,N\} \to \{1,\hdots,N\}$, we have that 
\[
J_i^{N}(\gamma_1, \hdots, \gamma_N)= J_{\sigma(i)}^{N}(\gamma_{\sigma(1)}, \hdots, \gamma_{\sigma(N)}).
\]
%
Let us recall the classical notion of Nash equilibrium when applied to the game defined by the $N$ players,  the action set $\A$ (which is the same for all players) and the individual costs $(J_{i})_{i=1}^{N}$.

\begin{definition}\label{nash_equilibrium}We say that $(\gamma_{1,N}, \hdots, \gamma_{N,N})  \in \A^{N} $ is a \emph{Nash equilibrium} in distributed open-loop strategies if 
\be\label{Nash_equilibrium_inequality_0}
J^N(\gamma_{i,N}; (\gamma_{j,N})_{j \neq i}) \leq J^N(\gamma; (\gamma_{j,N})_{j \neq i})  \hspace{0.3cm} \forall \; \gamma \in \A, \; \; i=1, \hdots, N.
\ee
\end{definition}
\bigskip


The existence of a Nash equilibrium for this symmetric game is not immediate because of the lack of compactness of the set $\A$.  The following simple result shows that the time derivatives of the strategies in a Nash equilibrium configuration (provided that it exists) enjoy a uniform boundedness  property in $L^q((0,T);\RR^d)$.  This fact will allow us to reduce the set of action strategies in Definition \ref{nash_equilibrium}.
%
\begin{lemma}\label{bounds_l2} Assume that {\bf(A1)} holds and that $(\gamma_{1,N}, \hdots, \gamma_{N,N}) \in \A^N$ is a Nash equilibrium.  Then, there exists a constant $C>0$, independent of $N$, such that
\be\label{bounded_second_order_moment}
\int_{0}^{T} |\dot{\gamma}_{i,N}^{x}(t)|^q \dd t \leq C \hspace{0.4cm} \mbox{{\rm for $m_0$-a.e.}  $x \in \RR^d$}, \; \; i=1, \hdots, N.
\ee
\end{lemma}
\begin{proof} Fix $i \in \{1, \hdots, N\}$ and define $j_i: W^{1,q}([0,T];\RR^d) \to \RR$ as 
\[
	j_i(\gamma)= \int_{0}^{T}  \ell_i( \dot{\gamma}(t),\gamma(t), t) \dd t +  \Phi_i(\gamma(T)),
\]
where $\ell_i: \RR^{d} \times \RR^{d} \times [0,T] \to \RR$ and $\Phi_i:\RR^d \to \RR$ are defined by
\be\label{definition_hat_f}
\ba{rcl}
\ell_i(\alpha,x, t)&:=&\int_{\RR^{(N-1) \times d}}  \ell\left(\alpha, x, (\gamma_{j,N}^{x_j}(t))_{j\neq i}  \right) \otimes_{j\neq i}  \dd m_0(x_j),\\[6pt]
\Phi_i(x)&:=& \int_{\RR^{(N-1) \times d}} \Phi\left( x , (\gamma_{j,N}^{x_j}(T))_{j\neq i}  \right) \otimes_{j\neq i}  \dd m_0(x_j). \ea
\ee
By  Fubini's theorem, and \eqref{Nash_equilibrium_inequality_0}, for all $\gamma \in \A$ we have
\be\label{inequality_in_mean_wrt_m_0}
J^{N}(\gamma_{i,N};  (\gamma_{j,N})_{j\neq i})= \int_{\RR^d} j_{i}(\gamma_{i,N}^{x}) \dd m_0(x)\leq J^{N}(\gamma;  (\gamma_{j,N})_{j\neq i}) = \int_{\RR^d} j_{i}(\gamma^{x}) \dd m_0(x).
\ee
For $x \in \RR^d$ define 
$$\SS_i^N(x):=  \mbox{argmin}\left\{ j_i(\gamma) \; | \; \gamma \in \A(x)\right\} \subseteq \Gamma.$$
Assumption {\bf(A1)} implies that $j_i$ is lower semicontinuous, w.r.t.\ the weak topology in $W^{1,q}([0,T]; \RR^d)$ (see e.g. \cite[Corollary 3.24]{MR990890}). Using this fact,  the direct method in the Calculus of Variations and the first inequality in \eqref{bounds_for_ell}, we obtain that $\SS_i^N(x) \neq \emptyset$ for all $x\in \RR^d$.   \\
{\it Claim: } The set-valued map $\RR^d \ni x \rightrightarrows \SS_i^N(x) \in 2^{\Gamma}$   takes closed values and  is upper semicontinuous, i.e. for all closed sets $M  \subseteq \Gamma$  we have that $\{ x\in \RR^d \; | \; \SS_i^N(x) \cap M \neq \emptyset\}$ is closed. \\
Let us assume for a moment  that the claim is true. Then, as a consequence of its second statement, the set valued map $\SS_i^N$ is Borel measurable and, hence, by the  Kuratowski-Ryll-Nardzewski Selection Theorem (see e.g. \cite[Theorem 18.13]{MR2378491}) we have the existence of a Borel measurable function $\RR^d \ni x \mapsto \hat{\gamma}(x) \in \Gamma$ such that $\hat{\gamma}(x) \in \SS_i^N(x)$ for all $x\in \RR^d$. Since $\hat{\gamma} \in \A$, relation \eqref{inequality_in_mean_wrt_m_0} implies that $\gamma_{i,N}^{x} \in \SS_i^N(x)$ for $m_0$-almost every $x\in \RR^d$. Taking $x\in \RR^d$ such that $\gamma_{i,N}^{x} \in \SS_i^N(x)$, we have that $j_{i}(\gamma_{i,N}^{x}) \leq j_{i}(\bar{\gamma}^x)$, where $\bar{\gamma}^x(t):=x$ for all $t\in [0,T]$. Assumption ${\bf(A1)}$ implies that $j_{1}(\bar{\gamma}^x)$ is bounded by a constant, which is uniform for $x\in \mbox{supp}(m_0)$ and  independent of $N$. Using this fact, the first inequality in \eqref{bounds_for_ell} easily yields \eqref{bounded_second_order_moment}. 

It remains to prove the claim. First note that if $(\gamma_n^{x})$ is a sequence in $\SS_i^N(x)$ converging to $\gamma^x$ uniformly in $[0,T]$, then, by the first inequality in \eqref{bounds_for_ell}, the sequence $(\dot{\gamma}_n^x)$ is bounded in $L^{q}([0,T]; \RR^d)$. If $g$ is a weak limit point of  $(\dot{\gamma}_n^x)$ in $L^{q}([0,T]; \RR^d)$, then, passing to the limit along a subsequence  in the relation $\gamma_n^x(t)=x + \int_{0}^{t}\dot{\gamma}_n^x(s) \dd s$ for all $t \in [0,T]$, we get that $\gamma^{x} \in  W^{1,q}([0,T];\RR^d)$, with $\dot{\gamma}^x=g$, and, hence, the whole sequence $(\dot{\gamma}_n^x)$ converges weakly to $\dot{\gamma}^x$  in $L^{q}([0,T]; \RR^d)$. Using this fact and the weak lower semicontinuity of $j_{i}$ in $W^{1,q}([0,T]; \RR^d)$, we obtain that $\gamma^x \in \SS_i^N(x)$ and, hence, $\SS_i^N(x)$ is closed in $\Gamma$. In order to show that $\SS_i^N$ is upper semicontinuous, let $M$ be a closed subset of $\Gamma$ and let $(x_n)$ be a sequence in  $\{ x \in \RR^d \; | \; \SS_i^N(x) \cap M \neq \emptyset\}$ converging to some $\bar{x}\in \RR^d$. Then, by definition, there exists $\gamma_n^{x_n} \in \SS_{i}^N(x_n)\cap M$. Arguing as before, using the first inequality in \eqref{bounds_for_ell}, we obtain that $(\dot{\gamma}_n^{x_n})$ is bounded in $L^{q}([0,T]; \RR^d)$. This implies that, up to some subsequence, $\gamma_n^{x_n}$ converge uniformly to some $\gamma^{\bar{x}} \in W^{1,q}([0,T];\RR^d)$ and  $\dot{\gamma}^{x_n} \to \dot{\gamma}^{\bar{x}}$ weakly in $L^{q}([0,T]; \RR^d)$. Using that $M$ is closed, w.r.t.\ the uniform convergence, we get that $\gamma^{\bar{x}} \in M$. On the other hand, noticing that 
\[
j_{i}(\gamma_n^{x_n}) \leq  j_{i}(\gamma^{x_n}) \hspace{0.3cm} \forall \; \gamma^{x_n} \in \A(x_n),
\]
and the fact that any $\gamma  \in \A(\bar{x})$ satisfies $\gamma -\bar{x} + x_n \in \A(x_n)$, the weak lower semicontinuity of $j_i$ yields that $j_{i}(\gamma^{\bar{x}}) \leq  j_{i}(\gamma ) \hspace{0.3cm} \forall \; \gamma  \in \A(\bar{x})$, i.e,  $\gamma^{\bar{x}} \in \SS_i^N(\bar{x})\cap M$, and, hence, $\bar{x} \in \{ x \in \RR^d \; | \; \SS_i^N(x) \cap M \neq \emptyset\}$, which implies that the latter set is closed.

\end{proof}
Now, we focus our attention on the existence of Nash equilibria for the described game. Using Lemma \ref{bounds_l2}, a reformulation of the set of admissible strategies and cost functionals of the $N$-players game will be useful. Let $C>0$ be given by Lemma \ref{bounds_l2} and define the set
\[
 Q_{C}:=\left\{ \gamma \in W^{1,q}([0,T];\RR^d) \; | \; \int_{0}^{T}|\dot{\gamma}(t)|^q \dd t \leq C, \; \;  \gamma(0) \in \mbox{supp}(m_0)\right\}.
\]
Since $\mbox{supp}(m_0)$ is compact, H\"older's inequality yields the existence of a compact set $K_C\subseteq \RR^d$ such that
\be\label{trajectories_on_a_fixed_compact_set}
  \gamma(t) \in 	K_{C} \quad \text{for all } t\in [0,T],\; \gamma \in Q_C.
\ee
Using this fact and arguing as in the proof of Lemma \ref{bounds_l2}, we have that $Q_C$ is compact as a subset of $\Gamma$, that is, when it is endowed with the topology of uniform convergence. 

Given $\gamma_{i} \in \A$, let us set $m_{i} := \gamma_{i} \sharp m_{0} \in \P(\Gamma)$. By \eqref{definition_push_forward}, for any profile of strategies $(\gamma_i)_{i=1}^{N} \in \A^N$,  the cost for player $i$ is given by
\be\label{cost_writen_in_terms_of_measures}
J^{N}(\gamma_i;  (\gamma_j)_{j\neq i})= \int_{\Gamma^{N}} j^N(\tilde{\gamma}_{i}; (\tilde{\gamma}_{j})_{j\neq i})  \otimes_{j=1}^{N} \dd m_{j} (\tilde{\gamma}_{j}).
\ee
This expression for the cost motivates a relaxation of the game with strategies in $\A$ by considering strategies taking values in $\P(\Gamma)$. Let us define the set  $\A_{rel}$  of  {\it randomized  distributed open-loop strategies} by
\be\label{relaxed_set_of_measures}
\A_{rel}:= \left\{ m \in \P(\Gamma) \; | \; e_0 \sharp m= m_0, \; \; \mbox{supp}(m) \subseteq Q_{C}  \right\},
\ee
where, for each $t\in [0,T]$,  $e_{t}\!: \Gamma \rightarrow \RR^{d}$ is defined by $e_t(\gamma):=\gamma(t)$.

\begin{lemma} \label{Lemma_relaxed_set_of_measures} The set $\A_{rel}$ is convex and compact as a subset of $\P(\Gamma)$. 
\end{lemma} 
\begin{proof} Convexity follows directly from the definition. On the other hand, since $\A_{rel}\subseteq \P(Q_{C})$ and $ \P(Q_{C})$ is compact as a subset of $\P(\Gamma)$ (because $Q_C$ is a compact subset of $\Gamma$),  it suffices to check that $ \A_{rel}$ is closed in $\P(\Gamma)$, but this follows directly from \cite[Proposition 5.1.8]{Ambrosiogiglisav} and the fact that $Q_C$ is closed.
%
%

\end{proof}
\begin{remark}\label{marginal_compactness} For later use, note that if $m\in \A_{rel}$, then $[0,T] \ni t \to e_t \sharp m \in \P_1(\RR^d)$ is well-defined and, by \eqref{monge_rubinsteint_form_d_1}, belongs to $C([0,T]; \P_1(\RR^d))$. Moreover, since  $ \mbox{{\rm supp}}(m) \subseteq Q_{C} $, we easily check that there exists $C'>0$, independent of $m \in \A_{rel}$, such that  
\[
	d_1(e_t \sharp m, e_s \sharp m) \leq C' | t-s|^{\frac{1}{q'}} \;   \hspace{0.3cm} \forall \; s, \; t \in [0,T].
\]
%
Therefore,  by \eqref{trajectories_on_a_fixed_compact_set}, \cite[Proposition 7.1.5]{Ambrosiogiglisav} and Lemma \ref{Lemma_relaxed_set_of_measures},  the set $\{ [0,T] \ni t \to e_t \sharp m \in \P_1(\RR^d) \; | \; m \in \A_{rel}\}$ is compact in $C([0,T]; \P_1(\RR^d))$.
\end{remark}

 Motivated by \eqref{cost_writen_in_terms_of_measures}, we introduce the   new relaxed game which has $\A_{rel}$ as set  of strategies for each player and,   given a strategy profile $(m_{j})_{j=1}^{N} \subseteq \A_{rel}^{N}$, the cost for player $i$  is   given by
$$ J_{rel}^{N}(m_{i};(m_{j})_{j \neq i}) :=\int_{Q_C^{N}} j^N(\gamma_{i}; (\gamma_{j})_{j\neq i})  \otimes_{j=1}^{N}  \dd m_{j} (\gamma_{j}). $$
Note that the this game is still symmetric.  In this framework, a profile of strategies $(m_{1,N}, \hdots, m_{N,N})\in \A_{rel}^{N}$ is  called a \emph{Nash equilibrium}  in randomized  distributed open-loop strategies if 
\be\label{Nash_equilibrium_inequality_1}
J_{rel}^N(m_{i,N}; (m_{j,N})_{j \neq i}) \leq J_{rel}^N(m; (m_{j,N})_{j \neq i})  \hspace{0.3cm} \forall \; m \in  \A_{rel}, \; \; i=1, \hdots, N.
\ee

\begin{proposition}\label{nash_equilibria_finite_game}
Under assumption {\bf(A1)}, there exists at least one Nash equilibrium in randomized  distributed open-loop strategies having the form $(m_{N}, \hdots, m_{N})\in \A_{rel}^{N}$.
\end{proposition}
\begin{proof}
It suffices to show the existence of a fixed point of the following set-valued map
$$\A_{rel} \ni \mu \rightrightarrows \SS^N(\mu):= \mbox{argmin}\left\{J^{N}_{rel}(\mu', \mu, \hdots, \mu) \; \big| \; \mu' \in \A_{rel}  \right\}\subseteq \A_{rel}.$$
First note that, as a consequence of \cite[Lemma 5.1.7]{Ambrosiogiglisav}, for every $\mu \in \A_{rel}$ the map $\A_{rel} \ni  \mu' \to J_{rel}^{N}(\mu', \mu, \hdots,\mu) \in \RR$ is lower semicontinuous. Thus, the compactness of $\A_{rel}$ yields that $\SS^N(\mu) \neq \emptyset$ for all $\mu \in \A_{rel}$.   A similar argument implies that $\SS^N(\mu)$ is closed for all $\mu \in \A_{rel}$. Notice also that $\SS^N(\mu)$ is convex for all $\mu \in \A_{rel}$. Let us show that $\SS^N$ is upper semicontinuous. Consider a closed set  $M\subseteq \A_{rel}$ and a sequence $(\mu_n)$ in  $\{ \mu \in \A_{rel} \; | \; \SS^N(\mu) \cap M \neq \emptyset\}$ converging to some $\bar{\mu} \in \A_{rel}$. By definition, there exists $\nu_n \in M$ such that 
\be\label{before_pasing_to_limit_upper_semicontinuity}
J_{rel}^{N}(\nu_{n} ; \mu_n, \hdots, \mu_n) \leq J_{rel}^{N}(\mu' ; \mu_n, \hdots, \mu_n)  \hspace{0.3cm} \forall \; \mu' \in  \A_{rel}.
\ee
Since $M$ is compact (because $\A_{rel}$ is compact), there exists $\bar{\nu} \in M$ such that, up to some subsequence, $\nu_n \to \bar{\nu}$ narrowly. By \cite[Theorem 3.2]{MR0233396} and {\bf(A1)} for all $\mu' \in \A_{rel}$ we have  that $J_{rel}^{N}(\mu' ; \mu_n, \hdots, \mu_n)$ converges to  $J_{rel}^{N}(\mu' ;\bar{\mu}, \hdots, \bar{\mu})$. On the other hand, by \cite[Theorem 3.23]{MR990890} and {\bf(A1)}, the function  $Q_C^{N} \ni (\gamma_1, \hdots, \gamma_N) \to j^N(\gamma_{1}; (\gamma_{j})_{j\geq 2}) \in \RR$ is lower semi-continuous, which, using \cite[Lemma 5.1.7]{Ambrosiogiglisav} again, implies that $J_{rel}^{N}(\bar{\nu} ; \bar{\mu}, \hdots, \bar{\mu}) \leq \liminf_{n \to \infty}J_{rel}^{N}(\nu_{n} ; \mu_n, \hdots, \mu_n) $. Therefore, passing to the limit in \eqref{before_pasing_to_limit_upper_semicontinuity} we obtain that 
$$  J_{rel}^{N}(\bar{\nu} ; \bar{\mu}, \hdots, \bar{\mu}) \leq J_{rel}^{N}(\mu' ;\bar{\mu}, \hdots, \bar{\mu}) \hspace{0.3cm} \forall \; \mu' \in  \A_{rel}, $$
i.e. $\bar{\nu} \in \SS^N(\bar{\mu}) \cap M$, which implies the closedness $\{ \mu \in \A_{rel} \; | \; \SS^N(\mu) \cap M \neq \emptyset\}$  and the upper semicontinuity of $\SS^N$. Using the properties above, the existence of a fixed point for $\SS^N$ follows from the Kakutani-Fan-Glicksberg fixed-point theorem (see e.g. \cite[Corollary 17.55]{MR2378491}).
\end{proof}

\begin{corollary}\label{existence_unique_solution_for_a_e_initial_condition_x} Let $(m_N, \hdots, m_{N})\in  \A_{rel}^{N}$ be a Nash equilibrium for the game defined by the cost $J_{rel}^{N}$ and the set of strategies $\A_{rel}$. Define $\hat{\ell}: \RR^d \times \RR^d \times [0,T] \to \RR$ and $\hat{\Phi}: \RR^d \to \RR$ by 
\be\label{redefinition_costs_using_fubini}
\ba{rcl}
\hat{\ell}(\alpha, x, t)&:=& \int_{Q_{C}^{N-1}} \ell\left(\alpha, x, (\gamma_{j}(t))_{j=2}^{N}\right) \otimes_{j=2}^{N}  \dd m_{N} (\gamma_{j}), \\[6pt]
 \hat{\Phi}(x)&:=& \int_{Q_C^{N-1}} \Phi\left(x, (\gamma_{j}(T))_{j=2}^{N}\right) \otimes_{j=2}^{N}  \dd m_{N} (\gamma_{j}),
\ea
\ee
and assume that for $m_0$-almost every $x\in \RR^d$  the optimization problem 
\be\label{optimization_problem_with_Fubini}
\inf\left\{\int_{0}^{T} \hat{\ell}(\dot{\gamma}(t), \gamma(t), t) \dd t + \hat{\Phi}(\gamma(T)) \; \big| \; \gamma \in Q_C, \; \gamma(0)=x \right\},
\ee
admits a unique solution. Then, there exists $\gamma_{N} \in \A$ such that $m_{N}= \gamma_{N} \sharp m_0$. Moreover, $\gamma_N$ is $m_0$-uniquely determined.

In particular, $(\gamma_N, \hdots, \gamma_N)\in \A^{N}$ is a Nash equilibrium in distributed open-loop strategies.
\end{corollary}
\begin{proof} Define $\widehat{j}: Q_{C} \to \RR$ by
\be\label{definition_hat_j}
\widehat{j}(\gamma):= \int_{0}^{T} \hat{\ell}(\dot{\gamma}(t), \gamma(t), t) \dd t + \hat{\Phi}(\gamma(T)).
\ee
By definition of Nash equilibrium and Fubini's theorem, for all $m \in \A_{rel}$ we have that
\be\label{nash_equilibrium_inequality_particular_case_proof}
J_{rel}^{N}(m_N; m_N, \hdots, m_{N})=\int_{Q_{C}}\widehat{j}(\gamma)\dd m_{N}(\gamma)\leq J_{rel}^{N}(m; m_N, \hdots, m_{N}) = \int_{Q_{C}}\widehat{j}(\gamma)\dd m(\gamma).
\ee
By the desintegration theorem (see e.g. \cite[Theorem 5.3.1]{Ambrosiogiglisav}), there exists a Borel family $\{m_{N}^{x} \; | \; x \in \RR^d\}\subseteq \P(Q_{C})$, such that $m_{N}^x\left(\{\gamma \in Q_{C} \; | \; \gamma(0)=x\}\right)=1$, for $m_0$-almost every $x \in \RR^d$, and 
\be\label{desintegration_optimal_measure}
\int_{Q_{C}}\widehat{j}(\gamma)\dd m_{N}(\gamma)= \int_{\RR^d} \int_{Q_C} \widehat{j}(\gamma)\dd m_{N}^{x}(\gamma) \dd m_0(x).
\ee
Define the set-valued function $\hat{\SS}^N: \RR^d \to 2^{\Gamma}$ by 
$$
\hat{\SS}^N(x):= \mbox{argmin}\left\{ \widehat{j}(\gamma) \; | \; \gamma \in \A(x), \; \; \int_{0}^{T}|\dot{\gamma}(t)|^q\leq C\right\}.$$
Arguing as in the proof of Lemma \ref{bounds_l2}, we have the existence of a Borel measurable selection $\RR^d \ni x \mapsto \gamma_{N}^x \in \hat{\SS}^N(x)$, which, by assumption, is $m_0$-uniquely determined. Moreover, by definition, $\gamma_N\in \A$. Now, let us define $\hat{m}:= \gamma_{N} \sharp m_0 \in \A_{rel}$. 
Then, by \eqref{nash_equilibrium_inequality_particular_case_proof}, taking $m= \hat{m}$, and \eqref{desintegration_optimal_measure}, we have that 
\[
	\int_{\RR^d}\left[ \int_{Q_C} \widehat{j}(\gamma)\dd m_{N}^{x}(\gamma)-\widehat{j}(\gamma_{N}^x) \right] \dd m_0(x)\leq 0.
\]
Since the integrand in the expression above is non-negative, by definition of $\gamma_{N}^x$,  we deduce that $\widehat{j}(\gamma_{N}^x)=  \int_{Q_C} \widehat{j}(\gamma)\dd m_{N}^{x}(\gamma)$ for $m_0$-a.e. $x\in \RR^d$, and, hence, $\widehat{j}(\gamma)= \widehat{j}(\gamma_{N}^x)$ for $m_{N}^x$-a.e. $\gamma  \in Q_{C}$. Since, by assumption, $\hat{\SS}^N(x)=\{\gamma_{N}^x\}$  for $m_0$-a.e.\ $x\in \RR^d$, we deduce that $m^{x}_{N}= \delta_{\gamma_N^x}$ for $m_0$-a.e.\ $x\in \RR^d$, hence $m_{N}=\gamma_{N} \sharp m_0$. The result follows.
\end{proof}
\begin{remark} An example of application of Corollary \ref{existence_unique_solution_for_a_e_initial_condition_x} is provided in Section \ref{PDE_N_players}  below.
\end{remark}

\section{Convergence to a mean field game equilibrium} \label{sect_mfg}
In this section, we study the limit behavior, as $N \to \infty$, of symmetric Nash equilibria  in randomized  distributed open-loop strategies. The existence of such Nash equilibria is ensured by Proposition \ref{nash_equilibria_finite_game}. We begin  by  defining the limit object, i.e.\ the MFG equilibrium. Then we will prove that any cluster point of the sequence $(m_N)$ is a MFG equilibrium. 

Let us define  $J: W^{1,q}([0,T];\RR^d) \times \P_1(\Gamma) \to \RR$ by
\[
	J(\gamma,m):= \int_{0}^{T}  \ell(\dot{\gamma}(t), \gamma(t),  e_{t} \sharp m)  \dd t + \Phi(\gamma(T), e_{T} \sharp m).
\]
It is straightforward to check that $m\in \P_1(\Gamma)$ implies that $e_{t} \sharp m \in \P_1(\RR^d)$ for all $t\in [0,T]$, which implies that $J$ is well-defined.
 
Following the terminology in \cite{mazanti_santambrogio_MFGs}, we consider next the notion of {\it Lagrangian MFG equilibrium} (see e.g. \cite{CH17,MR3644590,cardaliaguet_meszaros_santambrogio_2018,cannarsa_capuani_2018}).

\begin{definition}\label{MFG_equilibrium} We say that $m_{\ast} \in \P_1(\Gamma)$ is a  Lagrangian MFG equilibrium  if $e_0 \sharp m_{\ast}=m_0$ and 
\be\label{def_equilibrium_with_support}
\mbox{{\rm supp}}(m_{\ast}) \subseteq \left\{ \gamma \in W^{1,q}([0,T];\RR^d) \; | \; J(\gamma, m_{\ast}) \leq J(\gamma', m_{\ast}) \; \; \forall \; \gamma' \in W^{1,q}([0,T]; \RR^d), \; \gamma'(0)=\gamma(0) \right\}.
\ee

\end{definition} \smallskip 

Reasoning as in the proof of Lemma \ref{bounds_l2}, assumption   {\bf(A1)} implies that if $m_\ast$ is a Lagrangian MFG equilibrium, then $\mbox{supp}(m_{\ast}) \subseteq Q_{C}$ and, hence, $m_\ast \in \A_{rel}$. Thus, $m_{\ast}$ is a Lagrangian MFG if and only if $m_{\ast} \in \A_{rel}$, $e_0 \sharp m_{\ast}=m_0$ and 
\be\label{def_equilibrium_with_support_QC}
\mbox{{\rm supp}}(m_{\ast}) \subseteq \left\{ \gamma \in Q_{C} \; | \; J(\gamma, m_{\ast}) \leq J(\gamma', m_{\ast}) \; \; \forall \; \gamma' \in Q_C, \; \gamma'(0)=\gamma(0) \right\}.
\ee
We still denote  by $J$ the restriction of $J$ to $Q_{C} \times \A_{rel}$ and recall that $Q_C$, endowed with the topology of uniform convergence,  is a compact set. For later use, let us state the following simple results.
\begin{lemma}\label{some_basic_properties} Suppose that {\bf(A1)} holds. Then, the following assertions hold true:\smallskip\\
{\rm(i)} The relative topology on $\A_{rel}$, as a subset of $\P_1(\Gamma)$, coincides with the topology induced by the narrow convergence. \smallskip\\
{\rm(ii)} The function $J$ is lower semicontinuous in $Q_{C} \times \A_{rel}$.\smallskip\\
{\rm(iii)} For all $\gamma \in Q_{C}$, the function $J(\gamma, \cdot)$ is bounded, uniformly in  $\gamma$, and continuous in $\A_{rel}$. 
\end{lemma}
\begin{proof} Assertion {\rm(i)} follows from the fact that both topologies coincide on $\P(Q_{C})$, because $Q_C$ is a compact subset of $\Gamma$.  In particular, if $m_n \to m$ narrowly in $\A_{rel}$, then  for all $t\in [0,T]$ we have that
\be\label{convergence_in_P1}
e_t \sharp m_{n} \to e_t \sharp m \hspace{0.3cm} \mbox{in $\P_1(\RR^d)$}. 
\ee
%
Thus, assertion {\rm(ii)} follows from \eqref{convergence_in_P1}, {\bf(A1)} and the proof of \cite[Theorem 3.23]{MR990890}.  Assertion {\rm(iii)} follows directly from  \eqref{bounds_for_ell},  \eqref{convergence_in_P1} and dominated convergence.
\end{proof}
If $m_\ast$ is a Lagrangian MFG equilibrium, we will denote by $\{m_\ast^x \; | \; x\in \RR^d\}$ the $m_0$-uniquely determined Borel family of probability measures on $Q_{C}$ satisfying that $m_\ast^x(Q_C \setminus \A(x))= 0$ and  $\dd m_{\ast}(\gamma)= \dd m_{\ast}^x(\gamma) \otimes \dd m_0(x)$. The existence of such a family is ensured by the disintegration theorem (see e.g. \cite[Theorem 5.3.1]{Ambrosiogiglisav}). We have the following equivalent characterization of a Lagrangian MFG equilibrium.

\begin{proposition}\label{equivalent_reformulation_mfg_eq} The following assertions are equivalent: \smallskip\\
{\rm(i)} The measure $m_\ast \in \A_{rel}$ is a Lagrangian MFG equilibrium.\smallskip\\
{\rm(ii)} For $m_0$-a.e. $x\in \RR^d$ we have that 
\be\label{definition_SS}
\mbox{{\rm supp}}(m_\ast^x) \subseteq \SS(x):= \mbox{{\rm argmin}}\left\{J(\gamma',m^{\ast}) \; | \; \gamma' \in \A(x),  \; \; \int_{0}^{T}|\dot{\gamma}'(t)|^q\leq C\right\}. 
\ee
{\rm(iii)}  The following inequality holds true: 
\be\label{inequality_mfg_equilibrium}
	\int_{Q_{C}} J(\gamma,m_{\ast})\, \dd m_{\ast}(\gamma) \leq \int_{Q_{C}} J(\gamma,m_{\ast})\, \dd m(\gamma) \quad\text{for all } m\in \A_{rel}.
\ee
\end{proposition}
\begin{proof} Let us prove the equivalence between {\rm(i)} and {\rm(ii)}. Let $m_\ast \in \A_{rel}$ be a Lagrangian MFG equilibrium. If \eqref{definition_SS} does not hold, 
there exists $A \in \B(\RR^d)$, with $m_0(A) >0$,  such that $m_{\ast}^x( \SS(x)^c)>0$ for all $x \in A$. Define the set  $E:=\left\{ \gamma \in Q_{C} \; | \; \gamma(0) \in A, \; \gamma \in  \SS(\gamma(0))^{c} \right\}= e_{0}^{-1}(A) \cap \left\{ \gamma \in Q_{C} \; | \; \gamma \in \SS(\gamma(0))^{c} \right\}$. Arguing as in the proof of the claim in Lemma \ref{bounds_l2}, the set $\left\{ \gamma \in Q_{C} \; | \; \gamma \in \SS(\gamma(0))   \right\}$ is closed in $Q_C$, which implies that $E \in \B(Q_{C})$. 
Since $m_{\ast}(E)= \int_{A} m_{\ast}^{x}(\SS(x)^{c}) \dd m_0(x) >0$, we obtain a  contradiction with \eqref{def_equilibrium_with_support_QC}.  Conversely, suppose that {\rm(ii)} holds and that $m_{\ast}$ is not a Lagrangian MFG equilibrium. Then, defining $E':= \{ \gamma \in Q_C \; | \; \gamma \in \SS(\gamma(0))^c\}$, which is an open set and, hence, belongs to $\B(Q_C)$, we have that
$$
0<m_{\ast}(E')= \int_{\RR^d} \int_{Q_C} \II_{E'}(\gamma) \dd m_{\ast}^{x}(\gamma) \dd m_0(x), 
$$
which is impossible because \eqref{definition_SS} implies that the r.h.s. above is equal to $0$. 

Let us now prove that {\rm(ii)} $\Leftrightarrow$ {\rm(iii)}.  Notice  that   
\be\label{disintegration_i}
\int_{Q_{C}} J(\gamma,m_{\ast})\, \dd m_{\ast}(\gamma)=  \int_{\RR^d} \int_{Q_C}J(\gamma,m_{\ast}) \dd m_{\ast}^{x}(\gamma) \dd m_0(x).
\ee
Analogously, given $m \in \A_{rel}$, we disintegrate it w.r.t. $m_{0}$ and write $\dd m(\gamma)= \dd m^x(\gamma) \otimes \dd m_0(x)$, where $\{m^x \; | \; x\in \RR^d\}$ is a $m_0$-uniquely determined Borel family  of probability measures on $Q_C$ such that $m^x(Q_C \setminus \A(x))= 0$ for $m_0$-a.e. $x\in \RR^d$.  Thus,  
\be\label{disintegration_ii}
\int_{Q_{C}} J(\gamma,m_{\ast})\, \dd m(\gamma)=  \int_{\RR^d} \int_{Q_C}J(\gamma,m_{\ast}) \dd m^{x}(\gamma) \dd m_0(x).
\ee
If  {\rm(ii)} holds, then for $m_0$-a.e. $x\in \RR^d$ and $m_{\ast}^x$-a.e. $\gamma \in Q_C$ we have
\be\label{pointwise_inequality_in_x} J(\gamma,m_{\ast}) \leq J(\gamma',m^{\ast}) \hspace{0.5cm} \forall \; \gamma' \in \A(x) \cap Q_C.\ee
Integrating both sides of \eqref{pointwise_inequality_in_x}, first with respect to $ \dd m^{x}(\gamma')$ and then with respect to $\dd m_{\ast}^{x}(\gamma)$, and using \eqref{disintegration_i}-\eqref{disintegration_ii} we obtain \eqref{inequality_mfg_equilibrium}. Conversely,  using the notations introduced above, suppose that  \eqref{inequality_mfg_equilibrium} holds and let $\hat{\gamma} \in \A$ be a Borel measurable selection of $\SS$ (the existence of such selection can be justified arguing exactly as in the proof of Lemma \ref{bounds_l2}).
Then, taking the measure $m\in \A_{rel}$ defined by $\dd m(\gamma)= \dd \delta_{\hat{\gamma}^x}(\gamma) \otimes \dd m_0(x)$ in \eqref{inequality_mfg_equilibrium} and using  \eqref{disintegration_i}, we deduce that 
\[
	\int_{Q_C}J(\gamma,m_{\ast}) \dd m_{\ast}^{x}(\gamma)= J(\gamma^x,m_{\ast}) \hspace{0.3cm} \mbox{for $m_0$-a.e. } x\in \RR^d,
\]
and, hence, $m^x_\ast$-almost every $\gamma$ belongs to $\SS(x)$. The conclusion follows. 
\end{proof}
Now, consider the symmetric $N$-player game defined in Section \ref{n_player_game},  with randomized  distributed open-loop strategies, and let  $(m_{N}, \hdots, m_{N}) \in \A_{rel}^{N}$ be a symmetric equilibrium.  Our main result in this section, stated in the next theorem, shows that any limit point of this sequence is a Lagrangian MFG equilibrium.
%
%

\begin{theorem}\label{main_result}
Suppose that {\bf(A1)} holds. Then any limit point $m_{\ast}$ of $(m_{N})_{N\in \NN}$ {\rm(}there exist at least one{\rm)} is a Lagrangian MFG equilibrium. Moreover, if $(m_{N_k})_{k\in \NN}$ is a subsequence of $(m_{N})_{N\in \NN}$ converging to $m_{\ast}$, then  $\sup_{t\in [0,T]} d_1\left(e_t \sharp m_{N_k}, e_t \sharp m_{\ast}\right) \to 0$ as $k\to \infty$. 
\end{theorem}

\begin{proof} We only prove the first assertion, since, having this result, the second one follows directly from Remark \ref{marginal_compactness}. For $N\geq 2$, let $Y^{N}_{1},\ldots,Y^{N}_{N}$ be independent and identically distributed (i.i.d.) $Q_{C}$-valued random variables with common distribution $m_{N}$ defined on some probability space $(\Omega_{N},\mathcal{F}_{N},\mathbf{P}_{N})$. We denote by $\mathbf{E}_{N}$ the expectation with respect to $\mathbf{P}_N$. For $i\in \{1,\ldots,N \}$, let $\mu^{N,i}$ denote the (random) empirical measure of $Y^{N}_{1},\ldots,Y^{N}_{N}$ excluding $Y^{N}_{i}$, that is,
\[
	\mu^{N,i}(\omega) := \frac{1}{N-1} \sum_{j\neq i} \delta_{Y^{N}_{j}(\omega)} \in \P(Q_C) \hspace{0.5cm} \forall \;  \omega\in \Omega_{N}.
\]
Notice that $Y^{N}_{i}$ and $\mu^{N,i}$ are independent for every $i$, while $\mu^{N,1},\ldots,\mu^{N,N}$ are identically distributed (not independent in general) with common distribution depending on $N$. Moreover, for every $i\in \{1,\ldots,N \}$,
\be\label{representation_in_terms_of_Y_N}
	J_{rel}^N(m_{N}; m_{N}, \ldots, m_{N}) = \mathbf{E}_{N} \left[ J\left(Y^{N}_{i}, \mu^{N,i} \right) \right].
\ee
Let $(m_{N_{k}})_{k\in \NN}$ be a converging subsequence of $(m_{N})_{N\in \NN}$ with limit $m_{\ast}$ for some $m_{\ast} \in \A_{rel}$. The existence of such a subsequence follows from the compactness of $\A_{rel}$.  Let us prove that $(\mu^{N_k,1})_{k\in \NN}$ converges in distribution to the deterministic limit $m_\ast$.  Let $\mathcal{T} \subset C_{b}(Q_{C})$ be countable and measure determining (or separating). Thus, $\mathcal{T}$ is a countable collection of bounded continuous functions on $Q_{C}$ such that two probability measures $\nu, \tilde{\nu}\in \P(Q_{C})$ are equal whenever $\int \psi\, d\nu = \int \psi\, d\tilde{\nu}$ for all $\psi \in \mathcal{T}$. Observe that $\mathcal{T}$ can be chosen countable since $Q_{C}$ is a Polish space, hence separable, under the supremum  norm topology.

For  $\psi\in \mathcal{T}$ set
\begin{align*}
	& m^{N}_{\psi} := \int_{Q_{C}} \psi(\gamma)\, \dd m_{N}(\gamma),&  & v^{N}_{\psi} := \mathbf{E}_{N} \left[ \left( \int_{Q_{C}} \psi(\gamma)\, \dd \mu^{N,1}(\gamma) - m^{N}_{\psi} \right)^{2} \right].&
\end{align*}
By construction and symmetry, for every $i\in \{1,\ldots,N \}$,
\begin{align*}
	& m^{N}_{\psi} = \mathbf{E}_{N} \left[ \int_{Q_{C}} \psi(\gamma)\, \dd \mu^{N,i}(\gamma) \right] = \mathbf{E}_{N} \left[ \psi(Y^{N}_{i}) \right], &
	& v^{N}_{\psi} = \mathbf{E}_{N} \left[ \left( \frac{1}{N-1} \sum_{j\neq i} \psi(Y^{N}_{j}) - m^{N}_{\psi} \right)^{2} \right]. &
\end{align*}
Since $Y^{N}_{1},\ldots,Y^{N}_{N}$ are independent and the functions in $\mathcal{T}$ are bounded, it follows that
\begin{equation} \label{EqProofConvergence}
	v^{N}_{\psi} = \frac{1}{(N-1)^2} \sum_{j = 2}^{N} \mathbf{E}_{N} \left[ \left( \psi(Y^{N}_{j}) - m^{N}_{\psi} \right)^{2} \right] \stackrel{N\to \infty}{\longrightarrow} 0 \quad\text{for every } \psi\in \mathcal{T}.
\end{equation}
Moreover, considering the subsequence $(m_{N_k})_{k\in \NN}$, we have that
\be\label{convergence_of_the_means}
	m^{N_{k}}_{\psi} \stackrel{k\to \infty}{\longrightarrow} \int_{Q_{C}} \psi(\gamma)\, \dd m_{\ast}(\gamma) =: m^{\ast}_{\psi} \quad\text{for every } \psi\in \mathcal{T}.
\ee
Since $\P(Q_{C})$ is compact, the set  $\P(\P(Q_{C}))$ is also compact. Thus, there exists a sub-subsequence $(N_{k_l})_{l\in \NN}$ such that $(\mu^{N_{k_l},1})_{l\in \NN}$ converges in distribution to $\mu$ for some $\P(Q_{C})$-valued random variable $\mu$ defined on some probability space $(\Omega,\mathcal{F},\mathbf{P})$. We denote by $\mathbf{E}$ the expectation under $\mathbf{P}$.  Then, by the continuous mapping theorem (see e.g. \cite[Theorem 5.1]{MR0233396}) and \eqref{convergence_of_the_means}, we have that
\[
	\lim_{l\to\infty} v^{N_{k_l}}_{\psi} = \mathbf{E} \left[ \left( \int_{Q_{C}} \psi(\gamma)\, \dd\mu(\gamma) - m^{\ast}_{\psi} \right)^{2} \right] \quad\text{for every } \psi\in \mathcal{T}.
\]
On the other hand, thanks to \eqref{EqProofConvergence}, $\lim_{l\to\infty} v^{N_{k_l}}_{\psi} = 0$. It follows that for every $\psi\in \mathcal{T}$,
\[
	\int_{Q_{C}} \psi(\gamma)\, \dd \mu(\gamma) = m^{\ast}_{\psi} =  \int_{Q_{C}} \psi(\gamma)\, \dd m_{\ast}(\gamma)\quad \mathbf{P}\text{-almost surely.}
\] 
Since $\mathcal{T}$ is countable, we have the existence of $A\in \mathcal{F}$ such that $\mathbf{P}(A) = 1$ and for every $\omega\in A$,
\[
	\int_{Q_{C}} \psi(\gamma)\, \dd \mu_{\omega}(\gamma) = \int_{Q_{C}} \psi(\gamma)\, \dd m_{\ast}(\gamma) \quad \text{for all } \psi\in \mathcal{T}.
\] 
Since $\mathcal{T}$ is measure determining, we find that
\[
	\mu_{\omega} = m_{\ast} \quad \text{for $\mathbf{P}$-almost all } \omega \in \Omega.
\]
As we can always choose converging (sub-)subsequences, we deduce that $(\mu^{N_{k},1})_{k\in \NN}$ converges in distribution to $m_{\ast}$.   By independence of $Y^{N_k}_{1}$ and $\mu^{N_k,1}$, and using \cite[Theorem 3.2]{MR0233396}, we find that
\[
	\left( Y^{N_k}_{1}, \mu^{N_k,1} \right)\sharp \mathbf{P}_{N_{k}} = m_{N_k} \otimes \left( \mu^{N_k,1} \right)\sharp \mathbf{P}_{N_{k}} \stackrel{k\to\infty }{\longrightarrow} m_{\ast} \otimes \delta_{m_{\ast}}.
\]
This implies, thanks to \eqref{representation_in_terms_of_Y_N}, Lemma \ref{some_basic_properties}{\rm(ii)}   and \cite[Lemma 5.1.7]{Ambrosiogiglisav}, that  
\be\label{limit_lhs}
	\liminf_{k\to \infty} J_{rel}^{N_k}(m_{N_k}; m_{N_k}, \ldots, m_{N_k}) = \liminf_{k\to \infty} \mathbf{E}_{N_k} \left[ J\left(Y^{N_k}_{1}, \mu^{N_k,1} \right) \right] \geq \int_{Q_{C}} J(\gamma,m_{\ast})\, \dd m_{\ast}(\gamma).
\ee
Let $m\in \A_{rel}$. By Lemma \ref{some_basic_properties}{\rm(iii)} and dominated convergence, it follows that
\be\label{limit_rhs}
	J_{rel}^{N_k}(m; m_{N_k}, \ldots, m_{N_k}) = \int_{Q_{C}} \mathbf{E}_{N_k} \left[ J\left(\gamma, \mu^{N_k,1} \right) \right] \dd m(\gamma) \stackrel{k\to \infty}{\longrightarrow} \int_{Q_{C}} J(\gamma,m_{\ast})\, \dd m(\gamma).
\ee
Passing to the limit in the Nash equilibrium inequality 
$$	J_{rel}^{N_k}(m_{N_k}; m_{N_k}, \ldots, m_{N_k}) \leq J_{rel}^{N_k}(m; m_{N_k}, \ldots, m_{N_k}),$$ 
and using \eqref{limit_lhs}-\eqref{limit_rhs}, we obtain that \eqref{inequality_mfg_equilibrium} holds. The result now follows from Proposition \ref{equivalent_reformulation_mfg_eq}. 
%
%
\end{proof}

\begin{remark}\label{existence_proof_remark}
In particular, under assumption {\bf(A1)},  Theorem~\ref{main_result} ensures  the existence of at least one Lagrangian \mbox{MFG} equilibrium.
\end{remark}

\section{The first order Mean Field Game system} \label{sect_pde_approach}

In this section,  we first discuss, following \cite{CH17}, the relation between the notion of Lagrangian MFG equilibrium and the first order PDE system introduced by Lasry and Lions in \cite[Section 2.5]{LasryLions07} for some particular data. Next, in Section \ref{PDE_N_players}, we consider symmetric randomized distributed open-loop Nash equilibria for the $N$-player game, which, thanks to Corollary \ref{existence_unique_solution_for_a_e_initial_condition_x} and Assumption {\bf(A2)} below, can be identified with symmetric  distributed open-loop Nash equilibria (non-randomized).   Arguing as in the MFG limit, we connect these equilibria with a  first order PDE system which is similar to the one appearing in the limit case. Consequently,  analytic techniques can also be used in order study the limit behavior of these equilibria as the number of players tends to infinity.

Let $L: \RR^d \times \RR^d \to \RR$,  $f$, $g: \RR^d \times \P_{1}(\RR^d)\to \RR$ and $m_0 \in \P_{1}(\RR^d)$  satisfying that \medskip\\
{\bf(A2)}{\rm(i)} The function $L$ belongs to $C^2(\RR^d \times \RR^d)$, is bounded from below and \smallskip\\
{\rm(i.1)}   there exist $C_L>0$, $\overline{L}>0$ such that 
\be\label{uniform_quadratic_bound_from_above}
L(\alpha,x) \leq  \overline{L}|\alpha|^2 +C_L \hspace{0.4cm} \forall \;  \alpha, \; x \in \RR^d.
\ee
{\rm(i.2)} There exist  $c_L, \; c_{L}'>0$ such that 
\be\label{assumptions_derivatives_L}\ba{c}
\partial_{\alpha, \alpha}^{2} L(\alpha,x)(\alpha', \alpha') \geq c_{L} |\alpha'|^2 \hspace{0.4cm} \forall \  \alpha, \; \alpha' \in \RR^d, \; x\in \RR^d, \\[8pt]
| \partial_x L(\alpha,x) | \leq c_{L}'(1 + |\alpha|^2) \hspace{0.4cm} \forall  \; \alpha, \; x \in \RR^d.
\ea\ee
{\rm(ii)} The functions $f$ and $g$ are continuous. Moreover,  for every $m\in \P_1(\RR^d)$ the functions $f(\cdot,m)$ and $g(\cdot,m)$ belong to $C^2(\RR^d)$  and there exists a constant $C_{f,g}>0$ such that 
$$
\sup_{m \in \P_1(\RR^d)} \left\{\|f(\cdot,m)\|_{\C^2}  +\|g(\cdot,m)\|_{\C^2}\right\}  \leq C_{f,g}, 
$$
where, for $h=f$, $g$,  we have set 
$$\|h(\cdot,m)\|_{\C^2}:= \sup_{x\in \RR^d} \left\{ | h(x,m)| +\sum_{i=1}^{d}| \partial_{x_i}h(x,m)| + \sum_{i,j=1}^{d}| \partial_{x_i,x_j}h(x,m)|\right\}.$$
{\rm(iii)} The measure $m_0$ is absolutely continuous w.r.t. to the Lebesgue measure $\L^d$, with density still denoted by $m_0$,  and has a compact support.  \smallskip\\

A typical example of function $L$ satisfying {\bf(A2)}{\rm(i)} is given by $\RR^d \times \RR^d \ni (\alpha, x) \to L(\alpha,x):= b_1(x)|\alpha|^2+b_2(x)$, where, for $i=1$, $2$, $b_i \in C^2(\RR^d)$, $b_i$ is Lipschitz, and there exist constants $\underline{b}_1>0$, $\underline{b}_2 \in \RR$ and $\overline{b}_i>0$ such that $\underline{b}_i \leq b_i \leq \overline{b}_i$.
\begin{remark}{\rm(i)} Assumption {\bf (A2)}{\rm(i)} above implies the convexity of $L(\cdot,x)$ and the existence of $\underline{L}$, $C_L'$, $c^{''}_L>0$ such that 
\be\label{bounds_for_L_below}\ba{rcl}
L(\alpha, x) \geq \underline{L} |\alpha|^2 - C_L'  \hspace{0.4cm} \forall \;    \alpha, \; x  \in \RR^d, \\[6pt]
| \partial_\alpha L(\alpha,x) | \leq c_{L}^{''}(1 + |\alpha|) \hspace{0.4cm} \forall  \; \alpha, \; x \in \RR^d.
\ea
\ee
{\rm(ii)} For $x\in \RR^d$ let us denote by   $H(\cdot,x)$ the convex conjugate of $L(\cdot,x)$. The bound in ${\rm(i.1)}$ and the first bound in \eqref{bounds_for_L_below} imply the exitence  of constants $\underline{H}$, $\overline{H}$ and $C_H>0$ such that   
\be\label{growht_inequality_for_H_ast}
\underline{H}|\xi|^{2} - C_H \leq H(\xi,x) \leq \overline{H}|\xi|^{2} + C_H \hspace{0.3cm} \forall \;  \xi, \; x \in \RR^d.
\ee
{\rm(iii)} By the first estimate in \eqref{assumptions_derivatives_L} we have that $\partial_{\xi} H(\xi,x)$ is characterized  as the unique solution $\alpha(\xi,x)$   of the optimization problem $\max_{\alpha \in \RR^d}\{ \xi \cdot \alpha - L(\alpha,x)\}$. As a consequence of this fact and the first relation in \eqref{bounds_for_L_below}, we obtain the existence of $c_H>0$ such that  
\be\label{linear_growth_partial_H}|\partial_{\xi} H(\xi,x)| \leq c_H(1+|\xi|) \hspace{0.3cm} \forall \;  \xi, \; x \in \RR^d.\ee
Moreover, from the convexity of $L(x,\cdot)$, for all $\xi$, $x\in \RR^d$ we have that  $\partial_{\xi} H(\xi,x)$ is the unique solution to
\be\label{characterization_partial_H}
\partial_{\alpha} L(\partial_{\xi} H(\xi,x), x)= \xi.  
\ee
Using the relation above, the regularity $L \in C^2(\RR^d \times \RR^d)$, the first estimate in \eqref{assumptions_derivatives_L}  and the implicit function theorem, we  obtain that  $\RR^d \times \RR^d \ni (\xi,x) \to \partial_{\xi}H(\xi,x) \in \RR^d$ belongs to $C^1(\RR^d \times \RR^d; \RR^d)$. Using this fact, we get that $\RR^d \times \RR^d \ni (\xi, x) \mapsto \partial_{x} H(\xi,x)= - \partial_{x}L(\partial_{\xi}H(\xi,x),x) \in \RR^d$ is also of class $C^1$. As a consequence $H$ is of class $C^2$.  
\end{remark}
Let us define $\ell: \RR^d \times \RR^d \times \P_1(\RR^d) \to \RR$ and $\Phi : \RR^d \times \P_1(\RR^d) \to \RR$ by 
\be\label{data_regular_case}
\ell\left( \alpha,x,\mu\right):= L(-\alpha,x) + f(x, \mu) \hspace{0.3cm} \mbox{and } \; \; \Phi(x,\mu):=g(x, \mu).
\ee
Clearly, Assumption {\bf(A2)} implies that  $\ell$, $\Phi$ and $m_0$ satisfy {\bf(A1)}.  

\subsection{Lagrangian MFG equilibria and the MFG PDE system}\label{equivalence_both_mfg_formulations}

As pointed out in \cite{CH17}, under {\bf(A2)} the existence of a Lagrangian  equilibrium for the MFG problem defined  by $\ell$, $\Phi$ and $m_0$, is equivalent  the existence of a solution $(u, \rho )$ of the following PDE system, which was first introduced in \cite{LasryLions07},
$$\left.\ba{l}
- \partial_t u  +  H(\nabla u, x)  =  f(x,\rho(t)) \hspace{0.3cm} \mbox{in } \RR^d  \times (0,T), \\[6pt]
\partial_t \rho- \mbox{div}\left( \partial_{\xi} H( \nabla u, x) \rho \right)  = 0 \hspace{0.3cm} \mbox{in } \RR^d  \times (0,T),\\[6pt]
u(\cdot,T)= g(\cdot,\rho(T)), \; \; \rho(0)=m_0 \hspace{0.3cm} \mbox{in }   \RR^d . 
\ea \right\} \eqno(MFG)
$$
In the system above,  $u:\RR^d \times [0,T] \to \RR$ is a solution to the first equation, with the associated terminal condition, if it is globally Lipschitz,  locally semi-concave with respect to its first argument (see \cite[Section 2]{CannSinesbook}), uniformly in $t\in [0,T]$,  and the equation holds in the viscosity sense (see e.g.  \cite[Chapter III, Section 3]{BardiCapuzzo96}). In $(MFG)$,  $\nabla u$ denotes a Borel measurable selection of the set-valued map
$$
\RR^d \times [0,T] \ni (x,t) \mapsto D_{x}^{+}u(x,t):= \left\{p \in \RR^d \, | \; \limsup_{x' \to x} \frac{u(x',t)-u(x,t) - p\cdot (x'-x)}{|x'-x|} \leq 0\right\} \subseteq \RR^d. 
$$ 
The existence of such measurable selection follows from the fact that the above set-valued map has a closed graph (thanks to the semi-concavity property of $u$, see e.g. \cite[Proposition 3.3.4]{CannSinesbook}). Moreover, since $u$ is Lipschitz, $\nabla u(x,t)$ is uniformly bounded in $(x,t) \in \RR^d \times [0,T]$.  

The function $\rho: [0,T] \to \P(\RR^d)$ is a solution to the second equation, with the associated initial condition,  if $\rho \in C([0,T];\P_1(\RR^d))$, and the equation is satisfied in the sense of distributions, i.e.  for all $\phi \in C^{\infty}(\RR^d)$ with compact support we have
\be\label{definition_distributional_solution}
\int_{\RR^d} \phi(x) \dd \rho(t)(x)= \int_{\RR^d}  \phi(x) \dd m_0(x) - \int_{0}^{t}\int_{\RR^d} \partial_{\xi} H( \nabla u(x,s), x) \cdot \nabla \phi(x) \dd \rho(s)(x) \dd s.
\ee
Note that by the previous considerations, the second term in the right-hand-side of \eqref{definition_distributional_solution} is well-defined.

Any pair $(u,\rho)$ satisfying $(MFG)$  is  called an {\it equilibrium of the first order MFG problem}. 

For the sake of completeness, let us provide the main arguments that justify the equivalence between both notions of equilibria.  Let $m_{\ast} \in \A_{rel}$ be a Lagrangian MFG equilibrium.  Let us define $\rho: [0,T] \to \P(\RR^d)$ by $\rho(t):= e_t \sharp m_\ast$  for all $t\in [0,T]$, and {\it the value function} $u: \RR^d \times [0,T] \to \RR$ by 
\be\label{representation_soluion_HJB}
u(x,t) := \inf\left\{ \int_{t}^{T}  \ell(\dot{\gamma}(s), \gamma(s),  \rho(s))  \dd s + \Phi(\gamma(T),\rho(T)) \; \big| \; \; \gamma \in  W^{1,2}([t,T]; \RR^d)   \; \; \mbox{and } \;  \gamma(t)=x   \; \right\},
\ee
for all $(x,t) \in \RR^d \times [0,T]$.  Since $m_\ast \in \A_{rel}$, we have that  $\rho \in C([0,T];\P_1(\RR^d))$. Using this fact, assumption {\bf(A2)} and \cite[Proposition 1.1 and Remark 1.1]{MR1485734}, we obtain that $u$ is a viscosity solution of 
\be\label{viscosity_solution_hjb_equation}\ba{rcl}
- \partial_t u  +  H(\nabla u, x)  &=&  f(x,\rho(t)) \hspace{0.3cm} \mbox{in } \RR^d  \times (0,T), \\[6pt]
u(\cdot,T)&=& g(\cdot,\rho(T)) \hspace{0.3cm} \mbox{in }   \RR^d .
\ea
\ee
Moreover, by \cite[Theorem 2.1]{MR2784834}, the Hamilton-Jacobi-Bellman equation above admits a comparison principle, which implies that $u$ is its unique viscosity solution. We will need the following result, whose proof follows from standard arguments. 
\begin{lemma}\label{regularity_properties_value_function} Under {\bf(A2)} we have:\smallskip\\
{\rm(i)} For every $(x,t) \in \RR^d\times [0,T]$ the set $\SS(x,t)$ of paths $\gamma_\ast  \in W^{1,2}([t,T]; \RR^d)$ such that $\gamma_\ast (t)= x$ and 
\be\label{existence_of_a_solution_for_u_x_t}
u(x,t)= \int_{t}^{T}  \ell(\dot{\gamma}_\ast(s), \gamma_\ast(s),  \rho(s))  \dd s + \Phi(\gamma_\ast(T),\rho(T)) 
\ee
is non-empty.  Moreover, there exists a constant $C>0$, independent of $(x,t)$, such that 
\be\label{uniform_bound_optimal_velocities_limit_case}
\sup_{s \in [t,T], \; \gamma_\ast \in  \SS(x,t)} |\dot{\gamma}_\ast(s)| \leq C. 
\ee
{\rm(ii)} The value function $u$ is globally Lipschitz.\\
{\rm(iii)} The value function $u$ is locally semi-concave w.r.t. to the space variable, uniformly in $t\in [0,T]$. More precisely, for any compact set $K \subseteq \RR^d$ there exists a constant $C_K$, independent of $t$, such that for every $\lambda \in [0,1]$, $x$, $y \in K$, such that the segment $[x,y]$ is contained in $K$, the following inequality holds
\be\label{semiconcavity_inequality}
\lambda u(x,t) + (1-\lambda)u(y,t) \leq u( \lambda x + (1-\lambda)y) + C_K \frac{\lambda(1-\lambda)}{2}|x-y|^2.
\ee
\end{lemma}
\begin{proof} The proof being standard, we only sketch the main ideas. The fact that $\SS(x,t)$ is non-empty follows directly from {\bf(A2)} and the direct method in the Calculus of Variations. Moreover, by {\bf(A2)} and arguing as in the proof of Lemma \ref{bounds_l2} we obtain the existence of $c>0$, independent of $(x,t,\rho)$, such that  
\be\label{representation_soluion_HJB_with_c}\ba{ll}
u(x,t) = &\inf\left\{ \int_{t}^{T}  \ell(\dot{\gamma}(s), \gamma(s),  \rho(s))  \dd s + \Phi(\gamma(T),\rho(T)) \; \big| \; \; \gamma \in  W^{1,2}([0,T]; \RR^d), \; \gamma(t)=x, \right.\\[6pt]
  \;  & \left. \; \hspace{1cm}\mbox{and } \; \int_{0}^{T}|\dot{\gamma}(s)|^2 \dd s \leq c\right\}.\ea
\ee
Using this fact, the Euler-Lagrange equation associated to any element $\gamma_\ast \in \SS(x,t)$, the second estimate in \eqref{assumptions_derivatives_L} and arguing as in the proof of \cite[Theorem 6.2.5]{CannSinesbook}, we easily obtain \eqref{uniform_bound_optimal_velocities_limit_case}, which proves assertion {\rm(i)}. In order to prove {\rm(ii)}, notice that {\rm(i)} implies that the value function can also be written as
\be\label{representation_soluion_HJB_with_control} 
u(x,t) =  \inf\left\{ \int_{t}^{T}  \ell\left(\alpha(s), x+\int_{t}^{s}\alpha(s')\dd s',  \rho(s)\right)  \dd s + \Phi\left( x+\int_{t}^{T}\alpha(s')\dd s',\rho(T)\right) \; \big| \; \; \alpha \in \hat{A}_C\right\},
\ee
where $\hat{A}_C:= \left\{ \alpha \in  L^{\infty}([0,T]; \RR^d) \; | \;  \|\alpha\|_{L^{\infty}} \leq C\right\}.$ Using  the estimate $|\inf_{\alpha \in \hat{A}_C}   A(\alpha) - \inf_{\alpha \in \hat{A}_C} B| \leq \sup_{\alpha \in \hat{A}_C} |A(\alpha)-B(\alpha)|$ for any functions $A$, $B: L^{\infty}([0,T]; \RR^d) \to  \RR$, expression \eqref{data_regular_case}, the uniform Lipschitz property for $f$ and $g$ in {\bf(A2)}{\rm(ii)}, and the second estimate in \eqref{assumptions_derivatives_L}, we easily obtain that $u(\cdot, t)$ is globally Lipschitz, with a Lipschitz constant which is independent of $t \in [0,T]$. Similarly,  using \eqref{representation_soluion_HJB_with_control}  and the estimate \eqref{uniform_quadratic_bound_from_above}, we get that $u(x,\cdot)$ is globally Lipschitz, with a Lipschitz constant which is independent of $x\in \RR^d$. Assertion {\rm(ii)} follows. Finally, assertion {\rm(iii)} follows directly from \cite[Theorem 6.4.1]{CannSinesbook}.
\end{proof}

Now, let us consider the set-valued map
$$
\RR^d \ni x \mapsto \SS(x):= \mbox{argmin}\left\{   \int_{0}^{T}  \ell(\dot{\gamma}(t), \gamma(t),  \rho(t))  \dd t + \Phi(\gamma(T),\rho(T)) \; \big| \; \; \gamma \in  W^{1,2}([0,T]; \RR^d),   \; \;   \gamma(0)=x\right\}.
$$
Since Lemma \ref{regularity_properties_value_function}{\rm(ii)} implies that $u(\cdot, 0)$ is a.e. differentiable, \cite[Corollary 6.4.10]{CannSinesbook} yields that for a.e. $x\in \RR^d$ we have $\SS(x)=\{\tilde{\gamma}^x\}$ for some $\tilde{\gamma}^x \in \A(x)$.
Now, as in the proof of Proposition \ref{equivalent_reformulation_mfg_eq},  let $\gamma_{\ast} \in \A$ be a Borel measurable selection of $\SS$. Then,  for a.e. $x\in \RR^d$ we have that $\gamma_\ast^x=\tilde{\gamma}^x$. Thus, Proposition \ref{equivalent_reformulation_mfg_eq}{\rm(ii)}  yields  $m_{\ast}^x= \delta_{ \gamma_\ast^x }$ for a.e. $x \in \mbox{supp}(m_0)$ and, hence,   $m_\ast= \gamma_\ast \sharp m_0$. In particular,  $ \rho(t)= \gamma_{\ast}^{(\cdot)}(t) \sharp m_0$ for all $t\in [0,T]$. 

On the other hand, by \cite[Theorem 6.4.9, Theorem 6.3.3 and Theorem 6.4.8]{CannSinesbook}, for a.e. $x\in \RR^d$, we have   
\be\label{solution_characteristic_equations}
\dot{\gamma}_\ast^x(t)=- \partial_{\xi} H\left( \nabla u( \gamma_\ast^x(t),t),\gamma_\ast^x(t)\right) \hspace{0.3cm} \forall \; t \in (0,T), \; \; \gamma_\ast^x(0)=x,
\ee
where we underline that $u$ is differentiable w.r.t. to its first argument at the point $( \gamma_\ast^x(t),t)$  if $t\in (0,T)$ (see \cite[Theorem 6.4.7]{CannSinesbook}).   Denoting by still by $\nabla u$ a measurable selection of $ (x,t) \mapsto D_{x}^{+}u(x,t)$, for every $\phi \in C^{\infty}(\RR^d)$ with compact support and $t\in [0,T]$, we have    
$$\ba{rcl}
\int_{\RR^d} \phi(x)  \dd \rho(t)(x)&=&  \int_{\RR^d} \phi(\gamma_\ast^x(t))\dd m_0(x) \\[6pt]
\; &=&   \int_{\RR^d}\phi(x) \dd m_0(x) - \int_{\RR^d}\int_{0}^{t} \partial_{\xi} H\left( \nabla u( \gamma_\ast^x(s),s),\gamma_\ast^x(s)\right) \nabla \phi(\gamma_\ast^x(s))\dd s \dd m_0(x),\\[6pt]
\; &=&   \int_{\RR^d}\phi(x) \dd m_0(x) - \int_{0}^{t} \int_{\RR^d} \partial_{\xi} H\left( \nabla u(x,s),x\right) \nabla \phi(x)\dd\rho(s)(x) \dd s ,
\ea
$$ 
which implies that $\rho$ satisfies \eqref{definition_distributional_solution} and, hence, the couple $(u,\rho)$ solves $(MFG)$.
Notice that  under  {\bf(A2)} a Lagrangian MFG equilibrium $m_\ast$  exists (see Remark \ref{existence_proof_remark}) and, hence, the previous arguments show, in particular,  the existence of at least  one solution $(u,\rho)$ to $(MFG)$.

Conversely, if $(u,\rho)$ solves $(MFG)$, then the first equation therein implies that $u$ and $\rho$ are still related by \eqref{representation_soluion_HJB}.  
 By the second equation in $(MFG)$ and \cite[Theorem 8.2.1]{Ambrosiogiglisav}, there exists a probability measure $m_\ast \in \P(\Gamma)$ such that $\rho(t)= e_t \sharp m_\ast$ for all $t\in [0,T]$ and, considering the disintegration $\dd m_{\ast}(\gamma)= \dd m_{\ast}^x(\gamma) \otimes \dd m_0(x)$,   for a.e. $x\in \mbox{supp}(m_0)$ the support of the measure $m_{\ast}^x$ is contained in the set of solutions of \eqref{solution_characteristic_equations}.  By Lemma \ref{regularity_properties_value_function} and arguing as in the proof of \cite[Lemma 4.11]{Cardaliaguet10}, we have that every solution to \eqref{solution_characteristic_equations}   solves the optimization problem in the r.h.s. of \eqref{representation_soluion_HJB} with $t=0$. Thus, by Proposition \ref{equivalent_reformulation_mfg_eq}{\rm(ii)} we obtain that $m_\ast$ is a Lagrangian MFG equilibrium. Notice also that  $\SS(x)$ being a singleton for a.e. $x\in \RR^d$, the previous argument shows, in particular,  that $[0,T] \ni t \mapsto \gamma_\ast^{(\cdot)}(t) \sharp m_0 \in \P_1(\RR^d)$ is the unique solution in $C([0,T]; \P_1(\RR^d))$ of the continuity equation \\
 \be\label{limit_continuity_equation}
\partial_t \rho  - \mbox{div}\left( \partial_{\xi} H( \nabla u, x) \rho \right)  = 0 \hspace{0.3cm} \mbox{in } \RR^d  \times (0,T), \hspace{0.3cm} \rho (0)=m_0 \hspace{0.3cm} \mbox{in }   \RR^d.
\ee

In addition to the relation between Lagrangian MFG equilibria and the solutions of $(MFG)$, assumption {\bf(A2)} has also consequences on the regularity of the time  marginals $\{\rho(t)\; | \; t \in [0,T]\}$ as the following result shows. 
\begin{proposition}\label{regularity_result_for_the_marginals} In addition to {\bf(A2)}{\rm(iii)}, assume that $m_0 \in L^p(\RR^d)$ for some $p \in (1, +\infty]$ and let $(u,\rho)$ be a solution to $(MFG)$. Then, the following assertions hold true:\smallskip\\
{\rm(i)} There exists $c_1>0$, independent of $t\in [0,T]$,  such that $\supp(\rho(t)) \subseteq B(0,c_1)$ for all $t\in [0,T]$.\smallskip\\
{\rm(ii)} For all $t \in [0,T]$ the measure $\rho(t)$ is absolutely continuous w.r.t. the Lebesgue measure. Moreover, the density of $\rho(t)$, that we will still denote by $\rho(t)$, belongs to $L^{p}(\RR^d)$ and there exists a constant $c_2>0$, independent of $p\in (1,+\infty]$ and $t\in [0,T]$, such that 
\be\label{lp_estimates_continuity_equation}
\|\rho(t)\|_{L^{p}}\leq c_2 \|m_0\|_{L^{p}}.
\ee
\end{proposition}
\begin{proof} Assertion {\rm(i)} follows directly from the formula $\rho(t)= \gamma_\ast^{(\cdot)}(t) \sharp m_0$, where $\gamma_\ast^x \in \SS(x)$ for all $x\in \RR^d$,  Lemma \ref{regularity_properties_value_function}{\rm(i)} and the fact that $\supp(m_0)$ is compact. In order to prove {\rm(ii)}, let $\beta \in C^{\infty}(\RR^d)$, non-negative, with support contained in the unit ball and such that $\int_{\RR^d} \beta(x)\dd x=1$.  For $\eps>0$, let us define $\beta_{\eps}(x):= \eps^{-d}\beta(x/\eps)$, $u_\eps(x,t):= \left(\beta_\eps \ast u(\cdot, t)\right)(x)$ and consider the equation
\be\label{approximated_continuity_equation}
\partial_t \rho_\eps - \mbox{div}\left( \partial_{\xi} H( \nabla u_\eps, x) \rho_\eps \right)  = 0 \hspace{0.3cm} \mbox{in } \RR^d  \times (0,T), \hspace{0.3cm} \rho_\eps(0)=m_0 \hspace{0.3cm} \mbox{in }   \RR^d.
\ee
For every $x\in \RR^d$, let us define $\gamma_\eps^x \in C^{1}([0,T];\RR^d)$ as the unique solution to 
\be\label{solution_characteristic_equations_regular_flow}
\dot{\gamma}_\eps^x(t)=- \partial_{\xi} H\left( \nabla u_\eps( \gamma_\eps^x(t),t),\gamma_\eps^x(t)\right) \hspace{0.3cm} \forall \; t \in (0,T), \; \; \gamma_\eps^x(0)=x.
\ee 
By \cite[Proposition 8.1.8]{Ambrosiogiglisav}, equation \eqref{approximated_continuity_equation}  admits a unique solution in $C([0,T];\P_1(\RR^d))$, which is given by $\rho_\eps(t):= \gamma_\eps^{(\cdot)}(t)\sharp m_0$ for all $t\in [0,T]$. Moreover, by a standard change of variable argument (see e.g. \cite[Section 2]{MR2408257}), for every $t\in [0,T]$ we have that $\rho_\eps(t)$ is absolutely continuous, with density given by
$$
\rho_\eps(x,t)= \frac{m_0\left([\gamma_\eps^{(\cdot)}(t)]^{-1}(x)\right)}{\left| \mbox{det}\left( Y\left([\gamma_\eps^{(\cdot)}(t)]^{-1}(x),t\right) \right)\right|} \hspace{0.4cm} \mbox{for a.e. $x\in \RR^d$,}
$$
where, for each $y \in \RR^d$, $Y(y,\cdot)$ is defined as the unique solution to 
$$
\dot{Y}(t)= L_{\eps}(Y(t), t) \hspace{0.3cm} t \in (0,T), \; \; Y(0)=y,
$$
with $\RR^d \times [0,T] \ni (x,t) \mapsto L_\eps(x,t)\in \RR^{d\times d}$ being given by 
\be\label{definition_L_eps}
L_\eps(x,t):=D_x\left[ \partial_{\xi} H\left( \nabla u_\eps( x,t),x\right)\right]= \partial_{\xi, \xi}^2 H\left( \nabla u_\eps( x,t),x\right)\partial_{x,x}^2u_\eps(x,t)+ \partial_{\xi,x}^2 H\left( \nabla u_\eps( x,t),x\right).
\ee
Let us assume that $p \in (1,+\infty)$. By a change of variable again,  we obtain that 
\be\label{lp_estimate_in_terms_of_Y}
\|\rho_\eps(t)\|_{L^{p}}^p= \int_{\RR^d}m_0^p(x) \left| \mbox{det}\left( Y\left(x,t\right) \right)\right|^{1-p} \dd x= \int_{\supp(m_0)}m_0^p(x) \left| \mbox{det}\left( Y\left(x,t\right) \right)\right|^{1-p} \dd x. 
\ee
Now, for all $x\in \RR^d$ and $t \in [0,T]$, we have (see \cite[Section 2, estimate $(2.4)$]{MR2408257}) 
\be\label{estimate_for_the_Jacobien_power_p}\ba{rcl}
\left| \mbox{det}\left( Y\left(x,t\right) \right)\right|^{1-p} &\leq& \exp\left( (p-1)\int_{0}^{t}\left\| \left[\mbox{div}\left( - \partial_{\xi} H( \nabla u_\eps(\cdot,s), \cdot) \right)\right]_{-} \right\|_{L^{\infty}}\dd s\right)\\[6pt]
\; & \leq & \exp\left(p\int_{0}^{T}\left\| \left[\mbox{Tr}\left( L_\eps(x,t) \right)\right]_{+} \right\|_{L^{\infty}}\dd t\right),
\ea\ee
where $[a]_{-}:=\max\{0, -a\}$, $[a]_{+}=a+[a]_-$ and  for any $a\in \RR$, and $\mbox{Tr}\left( L_\eps(x,t) \right)$ denotes the trace of the matrix $L_\eps(x,t)$. On the other hand, Lemma \ref{regularity_properties_value_function}{\rm(ii)}, and the definition of $u_\eps$, imply that $\nabla u_\eps(x,t)$ is bounded, uniformly in $(x,t) \in \RR^d \times [0,T]$ and $\eps>0$. 
Moreover, by Lemma \ref{regularity_properties_value_function}{\rm(iii)}, the compactness of $\supp(m_0)$  and the definition of $u_\eps(\cdot,t)$ again, we can assume that  $u_\eps(\cdot,t)$ is  uniformly semiconcave in a bounded open set $\OO$ containing $\supp(m_0)$, i.e.  $u_\eps(\cdot,t)$ satisfies \eqref{semiconcavity_inequality} for all $x$, $y\in \OO$, with $C_K$ replaced by $\tilde{c}$, for some $\tilde{c}$ independent of $t$ and $\eps$ small enough.   By \cite[Proposition 1.1.3]{CannSinesbook}, we have that $\partial^2_{xx} u_\eps(x,t) - \tilde{c} I_d$ negative semidefinite for all $(x,t) \in \OO \times [0,T]$ and, hence, using that $\partial_{\xi,\xi} H(\xi,x)$ is positive semidefinite for all $\xi \in \RR^d$ and $x\in \RR^d$, there exists a constant $\hat{c}>0$, independent of $\eps$ and $t$, such that 
$$
L_\eps(x,t):=D_x\left[ \partial_{\xi} H\left( \nabla u_\eps( x,t),x\right)\right]= \partial_{\xi, \xi}^2 H\left( \nabla u_\eps( x,t),x\right)\partial_{x,x}^2u_\eps(x,t)+ \partial_{\xi,x}^2 H\left( \nabla u_\eps( x,t),x\right)- \hat{c} I_{d}
$$
is negative semidefinite for all $x\in \OO$. As a consequence, $\mbox{Tr}(L_\eps(x,t))$ is bounded from above by a constant which is independent of  $\eps>0$  small enough,  $x\in \OO$, and $t \in [0,T]$. Thus, by  \eqref{estimate_for_the_Jacobien_power_p} and taking the power $1/p$ in  \eqref{lp_estimate_in_terms_of_Y}, there exists $c_2>0$, independent of $\eps$, $t$ and $p$, such that 
\be\label{estimate_rho_eps}
\|\rho_\eps(t)\|_{L^{p}}\leq c_2 \|m_0\|_{L^{p}} \hspace{0.3cm} \forall \; t \in [0,T]. 
\ee
The previous estimate shows the existence of  $\tilde{\rho} \in L^{\infty}([0,T]; L^{p}(\RR^d))$ and a sequence $(\rho_{\eps_n})_{n\in \NN}$ such that, as $n\to \infty$, $\eps_n \to 0$ and  $\rho_{\eps_n} \to \tilde{\rho} \in L^{\infty}([0,T]; L^{p}(\RR^d))$ in the weak* topology.  By dominated convergence, we have that $\partial_{\xi} H( \nabla u_{\eps_n}, \cdot) \to \partial_{\xi} H( \nabla u, \cdot)$ in $L^{1}([0,T]; L^{s}(\RR^d))$ for any $s\in [1,+\infty)$. As a consequence,   $\tilde{\rho}$ satisfies  estimate \eqref{estimate_rho_eps} and, passing to the limit in \eqref{approximated_continuity_equation}, we get that the measure $[0,T] \ni t \to \tilde{\rho}(t) \L^{d}\in L^{p}(\RR^d)$ satisfies \eqref{definition_distributional_solution}. Using that $[0,T] \ni t \mapsto \rho(t) \in \P_1(\RR^d)$ is the unique solution to \eqref{limit_continuity_equation} in $C([0,T]; \P_1(\RR^d))$, \cite[Lemma 8.1.2]{Ambrosiogiglisav} implies that $\tilde{\rho}(t)  \L^{d}=\rho(t)$ for a.e. $t\in [0,T]$. Thus, for a.e. $t\in [0,T]$, $\rho(t)$ is absolutely continuous w.r.t. to the Lebesgue measure and estimate \eqref{lp_estimates_continuity_equation} holds for its density. Using this fact and that $\rho \in C([0,T];\P_1(\RR^d))$, the previous statement is valid in the whole time interval $[0,T]$, which proves {\rm(ii)} when $p<+\infty$. Since $c_2$ does not depend on $p$, assertion {\rm(ii)} for  $p=\infty$ follows by taking the limit in  \eqref{lp_estimates_continuity_equation} when $p\to \infty$.  
\end{proof}

\begin{remark} Similar regularization techniques have been recently employed in \cite{Dweilk_Mazanti_2019}, in order to establish $L^p$-estimates for the time evolving distributions describing  equilibria in optimal-exit MFGs. 
\end{remark}
\subsection{The $N$-player equilibria: associated time marginals and value functions} \label{PDE_N_players}

Let us consider the game with $N$ players  defined in Section \ref{n_player_game} with $\ell$ and $\Phi$ given by \eqref{data_regular_case}. Let $(m_N, \hdots,m_N) \in \A_{rel}^N$ be a symmetric  equilibrium in randomized  distributed open-loop strategies for the $N$-player game. Note that if for $h=f$, $g$ we define
\be\label{definition_h_N}
h_N(x,\mu) := \int_{(\RR^d)^{N-1}} h\left(x,  \frac{1}{N-1}\sum_{j=2}^{N} \delta_{x_{j}}\right) \otimes_{j=2}^{N}  \dd \mu (x_{j}) \hspace{0.5cm} \forall \; x\in \RR^d, \;  \mu \in \P_1(\RR^d),
\ee
we have that $f_N$ and $g_N$ satisfy the assumptions for $f$ and $g$ in {\bf(A2)}{\rm(ii)} (with the same constant $C_{f,g}$).  As a consequence of this fact,  the results in \cite[Chapter 6]{CannSinesbook} and Corollary \ref{existence_unique_solution_for_a_e_initial_condition_x} we obtain the existence of  $\gamma_N\in \A$ such that $m_N= \gamma_{N} \sharp m_0$, i.e.  $(m_N, \hdots,m_N)$ can be identified with the non-randomized symmetric equilibrium in distributed open-loop strategies given by $(\gamma_N, \hdots, \gamma_N) \in \A^N$. Furthermore, setting $\rho_N(t):= e_t \sharp m_N=\gamma_{N}^{(\cdot)}(t) \sharp m_N$ for all $t\in [0,T]$, we have that  $\gamma_{N}^x \in \SS^{N}(x):= \SS^N(x,0)$, where 
$$\ba{l}
 \SS^{N}(x,t):= \mbox{argmin}\left\{ \int_{t}^{T} \left[ L(-\dot{\gamma}(s),\gamma(s))+ f_{N}(\gamma(s),\rho_N(s)) \right]\dd s + g_{N}(\gamma(T),\rho_N(T)) \;   \big|  \right. \\[8pt]
\hspace{3.8cm} \left.  \gamma \in  W^{1,2}([t,T]; \RR^d) \; \; \mbox{and } \;  \gamma(t)=x\right\} \hspace{0.3cm} \forall \; x\in \RR^d, \; t\in[0,T].
\ea
$$
\begin{remark}\label{gamma_N_uniquely_defined} {\rm(i)}  Recall that the representation $m_N=\gamma_N\sharp m_0$ is only $m_0$-uniquely determined. In particular, if $\gamma_N' \in \A$ is different from $\gamma_N$ but coincides with it on a set $A$ such that $m_0(A)=1$,  then we also have that $m_N=  \gamma_N' \sharp m_0$.  For the sake of simplicity, we have chosen to represent always $m_N$ via a measurable selection $\gamma_N$ of the set-valued map $\SS^N$. Notice that {\bf(A2)} and the results in  \cite[Chapter 6]{CannSinesbook} imply that $\gamma_N^x$ is uniquely defined for a.e. $x\in \RR^d$. \smallskip\\
{\rm(ii)} Exactly as in the limit case {\rm(}see Lemma \ref{regularity_properties_value_function}{\rm(i))}, we have the existence of a constant $C>0$, independent of $(x,t)$ and $N \in \NN$, such that 
\be\label{uniform_bound_optimal_velocities_N_players}
\sup_{s \in [t,T], \; \gamma_\ast \in  \SS^N(x,t)} |\dot{\gamma}_\ast(s)| \leq C \hspace{0.3cm} \forall \; x\in \RR^d, \; t\in[0,T], \; \; N\in \NN.
\ee
As a consequence, there exists a compact set $K_C \subseteq \RR^d$ such that $\gamma_{N}^x(t) \in K_{C}$ for all $N \in \NN$, $x \in \supp(m_0)$ and $t \in [0,T]$. In particular, the representation $\rho_N(t)= e_t \sharp m_N=\gamma_{N}^{(\cdot)}(t)$ implies that $\supp(\rho_{N}(t)) \subseteq K_C$ for all $N \in \NN$ and $t\in [0,T]$. 
\end{remark} \smallskip

Let us define $u_{N}: \RR^d \times [0,T] \to \RR$ by
\be\label{representation_u_N}\ba{ll}
u_{N}(x,t):=&\inf\left\{ \int_{t}^{T} \left[ L(-\dot{\gamma}(s),\gamma(s))+ f_{N}(\gamma(s),\rho_{N}(s)) \right]\dd s + g_{N}(\gamma(T),\rho_{N}(T)) \;   \big|  \right. \\[8pt]
\; &  \; \; \; \; \; \;\; \; \; \; \; \; \left.    \gamma \in  W^{1,2}([t,T]; \RR^d) \; \; \mbox{and } \;  \gamma(t)=x\right\} \hspace{0.4cm} \forall \; x \in \RR^d, \; t\in [0,T].
\ea
\ee
\begin{remark}\label{on_the_uniform_regularity_properties_for_u_N}
Mimicking the proofs of Lemma \ref{regularity_properties_value_function} and  \cite[Theorem 6.4.1]{CannSinesbook} we obtain that $u_N$ is globally Lipschitz and locally semi-concave. Moreover, the  Lipschitz and local semi-concavity constants   are independent of $N$. 
\end{remark}
Arguing as in the previous subsection,  the pair $(u_N,\rho_N)$ solves
$$\left.\ba{l}
- \partial_t u_N   +  H(\nabla u_N , x)  =  f_{N}(x,\rho_N (t)) \hspace{0.3cm} \mbox{in } \RR^d  \times (0,T), \\[6pt]
\partial_t \rho_N - \mbox{div}\left( \partial_{\xi} H( \nabla u_N , x) \rho_N\right)  = 0 \hspace{0.3cm} \mbox{in } \RR^d  \times (0,T),\\[6pt]
u_N(\cdot,T)= g_{N}(\cdot,\rho_N(T)), \; \; \rho_N(0)=m_0 \hspace{0.3cm} \mbox{in }   \RR^d.
\ea \right\} \eqno(MFG_N)
$$ 
Conversely, associated to any solution $(u_N, \rho_N)$ we have the existence of a symmetric equilibrium in distributed open-loop strategies  $(\gamma_{N},\hdots, \gamma_N)\in \A^N$ for the $N$-player game. Moreover, using again the results in \cite[Chapter 6]{CannSinesbook}, any $\gamma_N\in \A$ defining such equilibrium satisfies 
\be\label{solution_characteristic_equations_N_players}
\dot{\gamma}_N^x(t)=- \partial_{\xi} H\left( \nabla u_N( \gamma_N^x(t),t),\gamma_N^x(t)\right) \hspace{0.3cm} \forall \; t \in (0,T), \; \; \gamma_N^x(0)=x,
\ee
for a.e. $x \in \mbox{supp}(m_0)$.  Thus, we can think of the r.h.s. above as an optimal control which is feedback with respect to the individual states. We call $(\gamma_N, \hdots, \gamma_N)$ a Nash equilibrium in {\it distributed Markov strategies} for the $N$-player game. As a consequence of the previous discussion, such equilibria exist for all $N\in \NN$ provided that {\bf(A2)} holds true.


Let us consider a sequence  $(\gamma_{N})_{N\in \NN}$ of elements in $\A$ defining Nash equilibria in distributed Markov strategies for the $N$-player games.
 Theorem \ref{main_result} yields the existence of a Lagrangian equilibrium $m_\ast \in \A_{rel}$ and a subsequence $(\gamma_{N_k})_{k\in \NN}$ such that $\gamma_{N_k}\sharp m_0 \to m_\ast$ as $k \to \infty$.  Let $\gamma_\ast \in \A$ be such that $m_\ast= \gamma_\ast \sharp m_0$ and $\gamma_\ast^x\in \SS(x)$ for all $x\in \RR^d$, i.e. an equilibrium  in distributed Markov strategies for the MFG.

Our aim now is to study the convergence of the associated time marginals $\rho_{N_k}$  to $\rho$, the convergence of the associated value functions $u_{N_k}$ to $u$ and, finally, the   convergence  of $\gamma_{N_k}^x$ to $\gamma_{\ast}^x$ for a.e. $x\in \RR^d$. 

We will need the following preliminary result.
\begin{lemma}\label{lemma_on_convergence_of_h_k}
Assume {\bf(A2)}{\rm(ii)} and let $K \subseteq \RR^d$ be a nonempty compact set. Consider a sequence of measures $(\mu_k)_{k \in \NN} \subseteq \P_1(\RR^d)$  such that $\mbox{{\rm supp}}(\mu_{k}) \subseteq K$, for all $k \in \NN$,  and, as $k\to \infty$,  $\mu_k \to \mu$ for some $\mu \in \P_1(\RR^d)$.  Then, for any sequence $(x_k)_{k \in \NN}$ and $x \in \RR^d$ such that $x_k \to x$, we have
\be\label{convergence_h_k}
h(x,\mu)=\lim_{k \to \infty}    h_{k}(x_k,\mu_{k}),
\ee
where $h=f$, $g$ and $h_k$ is defined by \eqref{definition_h_N}.
\end{lemma}
\begin{proof} Notice that {\bf(A2)}{\rm(ii)} implies that 
\be\label{eliminating_x_k_in_h_k}
|h_{k}(x_k,\mu_k)-h_{k}(x,\mu_k)|\leq  C_{f,g}|x-x_k|.
\ee
Now, let $Y^{k}_1, \hdots, Y^{k}_k$ be independent and identically distributed  $K$-valued random variables, defined on some probability space $(\Omega_k, \F_k, \mathbf{P}_k)$, with common distribution $\mu_k$. Using that $\P(\P(K))$ is compact and arguing as in the proof of Theorem \ref{main_result}, we obtain that, as $k\to \infty$, the $\P_1(\RR^d)$-valued random sequence $\left(\frac{1}{k-1}\sum_{j=2}^{k}\delta_{Y^{k}_j}\right)_{k \in \NN}$ converges in distribution to the deterministic measure $\mu$. Since \eqref{eliminating_x_k_in_h_k} can be written as 
\[
\left|h_{k}(x_k,\mu_k)-\mathbf{E}_{k} \left(h\left(x,\frac{1}{k-1}\sum_{j=2}^{k}\delta_{Y^{k}_j}\right)\right)\right|\leq  C_{f,g}|x-x_k|,
\]
relation \eqref{convergence_h_k} follows by letting $k\to \infty$.
\end{proof}
\begin{theorem}\label{convergence_result_regular_case} Assume  that {\bf(A2)} holds. Then, the following assertions hold true: \smallskip\\
{\rm(i)} The sequence $(\rho_{N_k})_{k\in \NN}$ converges to $\rho$ in $C([0,T]; \P_1(\RR^d))$. \smallskip\\
{\rm(ii)} The sequence $(u_{N_k})_{k\in \NN}$ converges   to  $u$ uniformly on compact subsets of $\RR^d \times [0,T]$. \\
{\rm(iii)} For a.e.  $x\in \RR^d$, the sequence $(\gamma^x_{N_k})_{k\in \NN}$ converges to $\gamma^x_{\ast}$ uniformly in $[0,T]$ and $(\dot{\gamma}^x_{N_k})_{k\in \NN}$ converges to $\dot{\gamma}^x_{\ast}$  in the weak* topology in $L^{\infty}([0,T];\RR^d)$. 
\end{theorem}
\begin{proof}
Assertion {\rm(i)} follows directly from Theorem \ref{main_result}. Note that {\bf(A2)} implies that $(u_{N_k})_{k \in \NN}$ is a sequence of uniformly bounded functions on $\RR^d \times [0,T]$. Let us fix $(x,t) \in \RR^d \times [0,T]$. The definition of $u_{N_k}$, Remark \ref{gamma_N_uniquely_defined}{\rm(ii)}  and Lemma \ref{lemma_on_convergence_of_h_k} imply that 
\be\label{lim_sup_for_un_k} \limsup_{k \to \infty} u_{N_k}(x,t) \leq u(x,t).
\ee
 Let   $\gamma_{N_k}^{x,t}\in \SS^N(x,t)$ and $\gamma^{x,t} \in C([t,T]; \RR^d)$ be a cluster point of $(\gamma_{N_k}^{x,t})_{k\in \NN}$, with respect to the uniform convergence. The existence of $\gamma^{x,t}$ is ensured by \eqref{uniform_bound_optimal_velocities_N_players} and the  Arzel\`a-Ascoli theorem.  Up to the extraction of a subsequence, we can assume that $ \liminf_{k\to \infty}u_{N_k}(x,t)=\lim_{k\to \infty}u_{N_k}(x,t)$ and $\lim_{k\to \infty}\gamma_{N_k}^{x,t}=\gamma^{x,t}$ in $C([t,T]; \RR^d)$. Using estimate \eqref{uniform_bound_optimal_velocities_N_players} again, we get that   $\dot{\gamma}^{x,t}$ exists and $\dot{\gamma}_{N_k}^{x,t} \to \dot{\gamma}^{x,t}$ in the weak* topology in $L^{\infty}([0,T];\RR^d)$. By the weak lower semi-continuity of the cost functional we obtain 
$$
u(x,t) \leq \int_{t}^{T} \left[ L(-\dot{\gamma}^{x,t}(s),\gamma^{x,t}(s))+ f(\gamma^{x,t}(s),\rho(s)) \right]\dd s + g(\gamma^{x,t}(T),\rho(T)) \leq \liminf_{k\to \infty}u_{N_k}(x,t).
$$
Thus, by \eqref{lim_sup_for_un_k} we get the pointwise convergence  
$$
\lim_{k\to +\infty} u_{N_{k}}(x,t) = u(x,t) \hspace{0.4cm} \forall \; (x,t) \in \RR^d \times [0,T],
$$
and hence, using that $u_{N_k}$ is Lipschitz continuous, with a Lipschitz constant which is independent of $k$, assertion {\rm(ii)} follows from the Arzel\`a-Ascoli theorem. Finally, {\rm(iii)} is a consequence of the previous analysis with $t=0$ and the fact that \cite[Corollary 6.4.10]{CannSinesbook} implies that for a.e. $x\in \RR^d$ the optimization problem associated with $u(x,0)$ admits a unique solution.
\end{proof}
Recall that, as in the case of $(MFG)$, to each solution $(u_N,\rho_N)$ of $(MFG_N)$ we can associate a symmetric equilibrium  $(m_N, \hdots, m_N)\in \A_{rel}^N$  of the $N$-player game. As a consequence of this fact, Theorem \ref{main_result} and Theorem \ref{convergence_result_regular_case}, we have the following result. 
\begin{corollary}\label{Convergence_from_the_system} Let $((u_N, \rho_N))_{N \in \NN}$ be a sequence of solutions to $(MFG_N)$ {\rm(}$N \in \NN${\rm)}. Then, there exists a solution $(u,\rho) \in C([0,T];\P_1(\RR^d))$ to $(MFG)$ such that, up to some subsequence, $u_N \to u$ uniformly over  compact subsets of $\RR^d \times [0,T]$ and $\rho_N \to \rho$ in $C([0,T];\P_1(\RR^d))$.
\end{corollary}

\begin{remark} If $h=f$, $g$ satisfies
$$
\int_{\RR^d} \left(h(x,\mu)-h(x,\mu')\right) \dd (\mu-\mu')(x) \geq 0 \hspace{0.3cm} \forall \; \mu, \; \mu' \in \P_1(\RR^d),
$$
then the solution $(u,\rho)$ to $(MFG)$ is unique {\rm(}see \cite{LasryLions07} and \cite[Corollary 5.2]{hadikhanloo2017learning}{\rm)}. Since any Lagrangian equilibrium $m_\ast$ can be represented by $\gamma_\ast \sharp m_0$, where $\gamma_\ast^x \in \SS(x)$ is uniquely determined for a.e. $x\in \RR^d$, the Lagrangian equilibrium must  also be unique. In this case,  the results in Theorem \ref{convergence_result_regular_case} hold for the entire sequence $( u_N, \rho_N, \gamma_N)$ and the result in Corollary \ref{Convergence_from_the_system} holds for the entire sequence $( u_N, \rho_N)$. 
\end{remark}

Finally, let us point out that the convergence result in Corollary \ref{Convergence_from_the_system} can also be established directly, without appealing to Theorem \ref{main_result}, under a stronger regularity assumption than {\bf(A2)}{\rm(ii)}. Indeed, assume that, in addition to {\bf(A2)}, $m_0 \in L^p(\RR^d)$ for some $p \in (1, +\infty]$. If $(u_N,\rho_N)$ solves $(MFG_N)$, then by \cite[Theorem 8.2.1]{Ambrosiogiglisav} and the results in  \cite[Chapter 6]{CannSinesbook}, we must have that  $\rho_N(t)= \gamma_N^{(\cdot)}(t) \sharp m_0$, for some $\gamma_N \in \A$ such that $\gamma_N^x \in \SS^N(x)$ for all $x \in \supp(m_0)$. 
Arguing as in the proof of Lemma \ref{regularity_properties_value_function}  we get the existence of $C'>0$ such that 
$$\sup_{x \in \supp(m_0), \,  t\in [0,T]} |\dot{\gamma}_N^x(t)| \leq C' \hspace{0.3cm} \forall \; N\in \NN.$$
 Therefore, there exists $C>0$  and a compact set $K\subseteq \RR^d$, both independent of $N$, such that $\gamma_N^x(t) \in K$ for a.e. $x \in \supp(m_0)$  and all $t\in [0,T]$. In particular,   $\supp(\rho_{N}(t)) \subseteq K$ and $d_1(\rho_N(s), \rho_N(t)) \leq C|s-t|$ for  all $s$, $t\in [0,T]$ and $N\in \NN$. This implies the existence of $\rho\in C([0,T]; \P_1(\RR^d))$ such that, up to some subsequence, $\rho_N \to \rho$ in $C([0,T];\P_1(\RR^d))$ as $N \to \infty$. Since Lemma \ref{lemma_on_convergence_of_h_k}
 implies that $f_N(\cdot,\rho_N(\cdot))$ converges uniformly to $f(\cdot, \rho(\cdot))$ on compact subsets of $\RR^d \times [0,T]$, standard stability results for viscosity solutions of Hamilton-Jacobi-Bellman equations imply that,  up to some subsequence, $u_N \to u$ uniformly on compact subsets of $\RR^d\times [0,T]$,    $u$ being the unique viscosity solution to  \eqref{viscosity_solution_hjb_equation}. In particular,   $u_N$ being  locally semi-concave with respect to the space variable, uniformly in $N$, for all $t\in [0,T]$ we have that $\nabla u_N(x,t) \to \nabla u(x,t)$ for a.e.  $x \in \RR^d$. Using that  $u_N$ is globally Lipschitz,  uniformly in $N$, by dominated convergence we deduce that
 \be\label{pointwise_flow_convergence}\partial_{\xi} H\left( \nabla u_N( \cdot,\cdot),\cdot\right) \to \partial_{\xi} H\left( \nabla u( \cdot,\cdot),\cdot\right) \hspace{0.15cm}  \mbox{in $L^{1}([0,T]; L^{s}(\RR^d))$ for any $s\in [1,+\infty)$.}\ee
On the other hand,  using again the uniform local semiconcavity of $u_N(\cdot, t)$  and arguing as in the proof Proposition \ref{regularity_result_for_the_marginals} we get the existence of $c_3>0$, independent of $N$, such that 
\be\label{lp_estimates_continuity_equation_uniform_IN_n}
\|\rho_N(t)\|_{L^{p}}\leq c_3 \|m_0\|_{L^{p}} \hspace{0.5cm} \forall \; t\in [0,T]. 
\ee
Using this bound, we obtain that $\rho(t)$ is absolutely continuous w.r.t. the Lebesgue measure for a.e. $t\in [0,T]$, and its density,  denoted likewise by $\rho(t)$, satisfies $\|\rho(t)\|_{L^{p}}\leq c_3 \|m_0\|_{L^{p}}$ for a.e. $t\in [0,T]$. Since $\rho \in C([0,T]; \P_1(\RR^d))$, the previous bound implies that for all $t\in [0,T]$, the measure $\rho(t)$  is absolutely continuous w.r.t. the Lebesgue measure and the estimate $\|\rho(t)\|_{L^{p}}\leq c_3 \|m_0\|_{L^{p}}$ holds. Moreover, using \eqref{pointwise_flow_convergence}-\eqref{lp_estimates_continuity_equation_uniform_IN_n} we can pass to the limit in the second equation $(MFG_N)$ to obtain that the pair $(u,\rho)$ solves $(MFG)$.

\bibliographystyle{plain}

\end{document}